\documentclass[11pt]{article}
\usepackage{amsfonts}
\usepackage{mathrsfs}
\usepackage{bbm}
\usepackage{amssymb,amsmath,graphicx}
\usepackage{float} \usepackage[colorlinks=true]{hyperref}
\usepackage{authblk}
\usepackage[numbers,sort&compress]{natbib}

\hypersetup{urlcolor=blue, citecolor=red}

\allowdisplaybreaks

\title{Long-time behavior of solution to a chemotaxis system
with weakly singular sensitivity and logistic source}
\author{Xiangdong Zhao{\thanks{
E-mail: zhaoxd1223@163.com }}\\
{{\small  School of Mathematics, Liaoning Normal University, Dalian 116029, P.R. China } }
}

\date{}

\begin{document}

 \maketitle
\date{}
\newtheorem{theorem}{Theorem}
\newtheorem{definition}{Definition}[section]
\newtheorem{lemma}{Lemma}[section]
\newtheorem{proposition}{Proposition}[section]
\newtheorem{corollary}{Corollary}[section]
\newtheorem{remark}{Remark}
\renewcommand{\theequation}{\thesection.\arabic{equation}}
\catcode`@=11 \@addtoreset{equation}{section} \catcode`@=12
\maketitle{}
\begin{abstract}

This paper is concerned with the parabolic-elliptic chemotaxis system with weakly singular sensitivity
and logistic source:~$
u_t=\Delta u-\chi\nabla\cdot(\frac{u}{v^\alpha}\nabla v)
+ru-\mu u^2$, $0=\Delta v-v+u,$ under the homogeneous Neumann boundary in a smooth bounded convex domain $\Omega\subset\mathbb{R}^n$ for $n\ge 2$. where $\alpha\in(0,1)$ and $\chi,r,\mu>0$. If $\alpha\in(0,\frac{n+2}{2n})$ and $\mu>\mu_0$ with $\mu_0>0$ suitably large, we give the explicit expression of the upper bound for $u$ with respect to the coefficient $\mu$ after some time, without establishing the uniformly positive bound for $v$ from below. Furthermore, by dealing with the corresponding non-singular chemotaxis system via the transformation $z=v^{1-\alpha}$, it is proved that the solution $(u,v)$ converges to $(\frac{r}{\mu},\frac{r}{\mu})$ in $L^\infty$-norm as $t\rightarrow\infty$ if $\alpha\in(0,\frac{1}{2})$ and $\mu>\mu_\star$ sufficiently large, which is moreover enjoying exponential convergence when $\alpha\in(0,\frac{n+2}{n^2+4})$.

\begin{description}
\item[2010MSC:]  35B45; 35B40; 92C17
\item[Keywords:] Chemotaxis; Singular sensitivity; Long-time behavior
\end{description}
\end{abstract}
\renewcommand{\thefootnote}{}

\section{Introduction}
Chemotaxis is a spontaneous biological phenomena. 1970, Keller and Segel proposed the following classical chemotaxis system \cite{KS1970}
\begin{equation}\label{p1}
\begin{cases}
u_t=\Delta u-\chi\nabla\cdot({u}{\nabla v})+f(u), \\
\tau v_t=\Delta v-v+u,
\end{cases}
\end{equation}
with $\chi>0$ and $\tau\in\{0,1\}$, in which the cells (with density $u$) move towards the concentration gradient of a chemical
substance (with concentration $v$) produced by the cells themselves. The crucial fact to study the global dynamic behavior of classical solution to \eqref{p1} is that the estimate on $\nabla v$ can be obtained by $u$ due to the linear equation for $v$ so as to control the chemotactic term $-\chi\nabla\cdot (u\nabla v)$. Generally, initial data suitably small, or space dimensionality suitably low if $f(u)=0$, or strong damping exponent in source $f(u)\not\equiv 0$ is sufficient to ensure global dynamic behavior of solution to \eqref{p1}, referring to \cite{NSY1997,C2015,TW2007,MW2010,C2017} and the reference therein. Whereas, blow-up phenomena can be constructed when initial data suitably large or space dimensionality suitably large if $f(u)=0$, or weak damping effect in source $f(u)\not\equiv 0$ \cite{N1995,N2001,W2013,W2018}.

If cellular behaviors obeys the Weber-Fechner's law, Keller and Segel in 1971 indicated that subjective sensation is proportional to the logarithm of the stimulus intensity and introduced the following chemotaxis system with logarithm sensitivity \cite{KS1971}
\begin{equation}\label{p2}
\begin{cases}
u_t=\Delta u-\chi\nabla\cdot(\frac{u}{v}{\nabla v})+f(u), \\
\tau v_t=\Delta v-v+u,
\end{cases}
\end{equation}
 with $\chi>0$ and $\tau\in\{0,1\}$. Intuitively, the possible singularity in $\frac{\chi}{v}$ is the challenge to study of global dynamic behaviors of solution. Let us recall the results for the case of $f(u)=0$. The conservation of mass for $u$ and the quantitative lower estimate for solution $v$ ensure that the chemical signal $v$ satisfies the {\it a priori } uniformly positive bound from below, i.e.,
\begin{align}\label{i1}
v(x,t)\ge c_0 \int_\Omega u_0dx.
\end{align}
where $c_0=c_0(\Omega)>0$. This means that the singularity is in fact absent, and provides an intuitive and effective way to derive global boundedness of solution to \eqref{p2}by direct releasing the bounded chemotactic sensitive function $\frac{\chi}{v}$ from the cross-diffusion term $-\chi\nabla\cdot(\frac{u}{v}\nabla v)$. Generally, chemotactic sensitive coefficient $\chi>0$ properly large relative to $n,\alpha$ will ensure global dynamic behavior of classical solution to the system \eqref{p2}, seeing \cite{A2019,HK2025,FS2016,FWY2015,L2026,AKL2019,WY2018,LX2026} and the references therein.
When $\tau=0$, if $\chi>\frac{2n}{n-2}$ with $n\ge 3$, and the moment of order 2 for $u_0$ sufficiently small under radial assumption, there exists a finite time blow-up solution \cite{NS1997}. Recall a corresponding chemotasixs $\epsilon u_t=\Delta u-\chi\nabla \cdot(\frac{u}{v}\nabla v), v_t=\Delta v-v+u$,
if $n\ge 3$ and $\epsilon>0$ sufficiently small, global boundedness of solution has been proved if $\chi<\frac{n}{n-2}$ \cite{FS2018}, whereas, spontaneous emergence of arbitrarily large
values of $u$ has been derived if $\chi>\frac{n}{n-2}$ under radial assumption in \cite{W2022}.
If $f(u)=ru-\mu u^\gamma$ with $\gamma\ge 2$ and $r,\mu>0$, the proliferation-death mechanism in logistic source destroys the mass conservation of $u$ and then makes the singularity in chemotactic sensitivity $\frac{\chi}{v}$ may be arrived. This is the essential challenge different to the case of $f(u)=0$. A natural and effective method to study the global boundedness and long-time behavior of classical solution is to establish the uniformly positive lower bound for $v$, which indeed can be arrived if $\chi>0$ suitably small relative to $r>0$ and $\gamma\ge 2$ \cite{FWY2014,HW2021,ZZ2017,ZZ2019,CWY2016,LL2021}. When $\tau=1$ and $\gamma>1$, global boundedness of classical solution to \eqref{p2} can be obtained if $\chi\in(0,\min\{\frac{1}{2}, \frac{1}{\sqrt{2(n-1)}}\})$,  without establishing this effective estimate on $v$ \cite{Z2022}.

Now, let us recall the following chemotaxis system with {\it weakly} singular sensitivity
\begin{equation}\label{p3}
\begin{cases}
u_t=\Delta u-\chi\nabla\cdot(\frac{u}{v^\alpha}{\nabla v})+ru-\mu u^2, \\
\tau v_t=\Delta v-v+u,
\end{cases}
\end{equation}
with $\chi,r,\mu>0$ and $\alpha\in(0,1)$. Thanks to the weakly singular sensitivity $\frac{\chi}{v^\alpha}$ with $\alpha\in(0,1)$, global boundedness of classical solution can be obtained if $\mu>0$ suitably larger despite the potential singularity in chemotactic sensitive function \cite{Z2023,L2025,LK2025,K2025}. In addition, global long-time behavior of classical solution have been established for the corresponding parabolic-elliptic case, based on the necessary estimate of the uniformly positive lower bound for $v$ in \cite{K2026,ZMT2025}.

Whereas, a natural question is that {\it whether to establish the uniformly positive lower bound for $v$ is necessary to obtain long-time behavior of classical solution to \eqref{p1} or not.}
In this paper, we concern with the long-time behavior of classical solution to the parabolic-elliptic chemotaxis system with weakly singular sensitivity and logistic source:
\begin{equation}\label{p}
\begin{cases}
u_t=\Delta u-\chi\nabla \cdot(\frac{u}{v^\alpha}\nabla v)+ru-\mu u^2, & x\in\Omega,~~t>0,\\
\displaystyle  0=\Delta v-v+u,& x\in\Omega,~~t>0,\\
  \displaystyle \frac{\partial u}{\partial \nu}=\frac{\partial v}{\partial \nu}=0 ,& x\in\partial\Omega,~~t>0,\\
  \displaystyle u(x,0)=u_0(x),  &x\in\overline\Omega,
\end{cases}
\end{equation}
where $\chi,r,\mu>0$ and $\alpha\in(0,1)$. $\Omega\subset \mathbb{R}^n$ with $n\ge 2$ is a smooth bounded convex domain and the positive initial datum $u_0\in C^1(\bar\Omega)$.

 Let
\begin{equation}\label{c1}
\tilde\mu_0:=
\begin{cases}
2^\frac{\alpha}{1-\alpha}\chi^\frac{1}{1-\alpha}+4^\frac{2-\alpha}{1-\alpha}(\chi\alpha)^\frac{1}{1-\alpha},~~&n=2,3,\\
(\frac{n}{2})^\frac{\alpha}{1-\alpha}\chi^\frac{1}{1-\alpha}+2\big(2+\frac{4}{n-2}\big)^\frac{1}{2}n^\frac{1}{1-\alpha}
(\chi\alpha)^\frac{1}{1-\alpha},~~&n\ge 4,
\end{cases}
\end{equation}
and $\mu_0:=1+\tilde\mu_0$. At first, we establish the explicit expression of upper bound of $u$ with respect to the coefficient $\mu$.
\begin{theorem}\label{th1}
Let $n\ge 2$ and $\alpha\in(0,\frac{n+2}{2n})$ with $\chi,r>0$. If $\mu>\mu_0$, then the system \eqref{p} possesses a globally bounded classical solution. Moreover, there exists $t_0>0$ and $L_0=L_0(\chi,\alpha,r,\Omega)>0$ independent of $\mu$ such that
\begin{align}\label{th11}
\|u\|_{L^\infty(\Omega)}\le \frac{L_0}{\mu},~~t>t_0.
\end{align}
\end{theorem}

Via the transformation of $z=v^{1-\alpha}$ for $\alpha\in(0,1)$, we obtain the global boundedness and H\"{o}lder regularity of $u$ by dealing with the corresponding non-singular chemotaxis system of $(u,z)$ for $\alpha\in(0,\frac{1}{2})$, and furthermore the estimate of $\int_\Omega |\nabla u|^{q_0}dx$ with some $q_0>n$ for $\alpha\in(0,\frac{n+2}{n^2+4})$ based on the {\it a priori } estimate $\int_\Omega \frac{|\nabla v|^2}{v^2}dx+\int_\Omega \frac{u}{v}dx=|\Omega|$. By means of the qualitative estimate of $\|u\|_{L^\infty(\Omega)}$ with respect to the coefficient $\mu$, we can prove for $\alpha\in(0,\frac{1}{2})$ that
\begin{align*}
\int_\Omega \frac{|\nabla v|^2}{v^{2\alpha}}dx\le \frac{1}{\chi^2}\int_\Omega \frac{|\nabla u|^2}{u^2}dx
\end{align*}
when $\mu>0$ sufficiently large, which will ensures the following long-time behavior of classical solution to \eqref{p}.
\begin{theorem}\label{th2}
Let $n\ge 2$ and $\alpha\in(0,\frac{1}{2})$ with $\chi,r>0$. Then there exists some $\mu_\star>0$ such that
\begin{align}\label{th21}
\|u-\frac{r}{\mu}\|_{L^\infty(\Omega)}+\|v-\frac{r}{\mu}\|_{L^\infty(\Omega)}\rightarrow 0,~~as~t\rightarrow\infty,
\end{align}
provided $\mu>\mu_\star$. Moreover, if $\alpha\in(0,\frac{n+2}{n^2+4})$ with $\mu>\mu_\star$, there exists some $\eta_\star=\eta_\star(r)>0$ and $L_\star=L_\star(\chi,\alpha,r,\mu,\Omega)>0$ such that
\begin{align}\label{th22}
\|u-\frac{r}{\mu}\|_{L^\infty(\Omega)}+\|v-\frac{r}{\mu}\|_{L^\infty(\Omega)}\le L_\star e^{-\eta_\star t},~~t>0.
\end{align}
\end{theorem}
\begin{remark}\label{r1}
{\rm Recall from \cite{CWY2016,LL2021,ZMT2025,K2026} that the decreasing property in time of the functional $\int_\Omega (u-\frac{r}{\mu}-\frac{r}{\mu}\ln\frac{\mu u}{r})dx$ is crucial to obtain the long-time behavior of classical solution, which can be arrived via controlling $\int_\Omega \frac{|\nabla v|^2}{v^{2\alpha}}dx$ with $\alpha\in(0,1]$ by $\int_\Omega (u-a)^2dx$ with the aid of the uniformly positive bound for $v$ from below. While, we will deal with $\int_\Omega \frac{|\nabla v|^2}{v^{2\alpha}}dx$ for $\alpha\in(0,\frac{1}{2})$ by $\int_\Omega \frac{|\nabla u|^2}{u^2}dx$ based on the qualitative estimate of upper bound of $u(\cdot,t)$ with respected on the coefficient $\mu$ after some time, and then obtain the desired long-time behavior of solutions to \eqref{p}. }
\end{remark}

\section{Preliminaries}
Let $\tilde\mu_0$ be defined in \eqref{c1}. Then we will obtain the global boundedness of classical solution of \eqref{p} via the similar arguments like that in \cite{LK2025} for $n\ge 3$ with $\alpha\in(0,1)$ which is also valid for $n=2$. To avoid repetition, we give it without detailed proof.

\begin{lemma}\label{lem11}
If $n\ge 2$ and $\alpha\in(0,1)$ with $\chi,r>0$, then for $\mu>\tilde\mu_0$ the problem \eqref{p} possesses a uniquely global bounded classical solution, i.e., there exists some $L_1=L_1(\chi,\alpha,r,\mu,\Omega)>0$ such that
\begin{align}\label{lem111}
\|u\|_{L^\infty(\Omega)}\le L_1,~~t>0.
\end{align}
\end{lemma}

Now, we give some fundamental estimates on the global bounded solution $(u,v)$ as follows.
\begin{lemma}\label{lem13}
It holds that
\begin{align}\label{lem131}
\int_\Omega \frac{u}{v}dx+\int_\Omega \frac{|\nabla v|^2}{v^2}dx=|\Omega|,
\end{align}
and for $p>1$ that
\begin{align}\label{lem132}
\frac{1}{2}\int_\Omega u^p\frac{|\nabla v|^2}{v^2}dx+\int_\Omega \frac{u^{p+1}}{v}dx\le 2\int_\Omega|\nabla u^\frac{p}{2}|^2dx+\int_\Omega u^pdx.
\end{align}
In addition, if $p\ge 3$, for $\epsilon_1>0$ there exists some $L_3=L_3(p,\Omega)>0$ such that
\begin{align}\label{lem133}
\int_\Omega \frac{|\nabla v|^{2p}}{v^p}dx\le (L_2+\epsilon_1)\int_\Omega u^pdx+\frac{L_3}{\epsilon_1^{p-1}}\big(\int_\Omega udx\big)^p,
\end{align}
where $L_2=\big(\frac{4(p+1)^2}{p-2}\big)^p\big(\frac{2(p-1)}{p-2}\big)^\frac{p}{2}$.
\end{lemma}
{\bf Proof.}\
Multiplying \eqref{p}$_2$ by $\frac{1}{v}$ and integrating over $\Omega$ by part, we get
\begin{align*}
0&=\int_\Omega \frac{1}{v}(\Delta v-v+u)dx=\int_\Omega \frac{|\nabla v|^2}{v^2}dx+\int_\Omega \frac{u}{v}dx-|\Omega|,
\end{align*}
which concludes \eqref{lem131}. \eqref{lem132} is the direct result established in \cite[Lemma 2.2]{LK2025}. If $p\ge 3$ and $k\in(2,2p-2)$, recall \cite[Proposition 3.1]{K2025} to know
\begin{align}\label{131}
\int_\Omega \frac{|\nabla v|^{2p}}{v^k}dx\le C_1\int_\Omega \frac{u^p}{v^{k-p}}dx+C_2\int_\Omega v^{2p-k}dx,
\end{align}
with $C_1=C_1(p,k)=\big(\frac{4(p-1)^2}{2p-k-2}\big)^p\big(\frac{2(k-1)}{k-2}\big)^\frac{p}{2}$ and some $C_2=C_2(p,k,\Omega)>0$. By the Young inequality with \eqref{p}$_2$, we get for $p>1$ that
\begin{align}\label{132}
0&=-(p-1)\int_\Omega v^{p-2}|\nabla v|^2dx-\int_\Omega v^pdx+\int_\Omega uv^{p-1}dx\nonumber\\
&\le -\frac{4(p-1)}{p^2}\int_\Omega|\nabla v^\frac{p}{2}|^2dx +\int_\Omega u^pdx.
\end{align}
In addition, if $p>2$, it is known by Ehrling's Lemma and the Young inequality that for $\epsilon_0>0$ there exists some $C_3=C_3(\Omega)>0$ such that
\begin{align*}
\int_\Omega v^pdx&\le \epsilon_0\int_\Omega |\nabla v^\frac{p}{2}|^2dx+\frac{C_3}{\epsilon_0}\big(\int_\Omega v^\frac{p}{2}dx\big)^2\\
&\le\epsilon_0\int_\Omega |\nabla v^\frac{p}{2}|^2dx+ \frac{1}{2}\int_\Omega v^pdx+\frac{2^{p-2}C_3^{p-1}}{\epsilon_0^{p-1}}\big(\int_\Omega vdx\big)^p.
\end{align*}
which along with $\int_\Omega vdx\le \int_\Omega udx$ entails that
\begin{align}\label{133}
\int_\Omega v^pdx\le 2\epsilon_0\int_\Omega |\nabla v^\frac{p}{2}|^2dx+\frac{(2C_3)^{p-1}}{\epsilon_0^{p-1}}\big(\int_\Omega udx\big)^p.
\end{align}
Then, we obtain by \eqref{131} for $k=p$ with \eqref{132} and \eqref{133} that
\begin{align*}
\int_\Omega \frac{|\nabla v|^{2p}}{v^p}dx&\le L_2\int_\Omega {u^p}dx+C_2\int_\Omega v^{p}dx\\
&\le (L_2+\frac{C_2p^2}{2(p-1)}\epsilon_0)\int_\Omega u^pdx+\frac{C_2(2C_3)^{p-1}}{\epsilon_0^{p-1}}\big(\int_\Omega udx\big)^p ,
\end{align*}
where $L_2=\big(\frac{4(p+1)^2}{p-2}\big)^p\big(\frac{2(p-1)}{p-2}\big)^\frac{p}{2}$. This concludes \eqref{lem133} with $\epsilon_1=\frac{C_2p^2}{2(p-1)}\epsilon_0$ and some $L_3=L_3(p,\Omega)>0$. Consequently, the proof is complete.   \qquad$\Box$\medskip

\section{Boundedness of $u$ with respect to $\mu$}
This section will establish the qualitative estimate of $\|u\|_{L^\infty(\Omega)}$ with respect to $\mu$. It should be mentioned here that the constants expressed in this section are all independent of $\mu$. Now, we deal with $\|u\|_{L^1(\Omega)}$.
\begin{lemma}\label{lem21}
If $\alpha\in(0,1)$ and $\mu>\tilde\mu_0$, then there exists some $t_1>0$ such that
\begin{align}\label{lem211}
\|u\|_{L^1(\Omega)}\le \frac{2r}{\mu},~~t>t_1.
\end{align}
\end{lemma}
{\bf Proof.}\
Integrate \eqref{p}$_1$ over $\Omega$ by part to have
\begin{align*}
\frac{d}{dt}\int_\Omega udx&\le r\int_\Omega udx-\mu \int_\Omega u^2 dx \le r\int_\Omega udx -\frac{\mu}{|\Omega|}\Big(\int_\Omega udx\Big)^2,~~t>0
\end{align*}
by the H\"{o}lder inequality, which along with the Bernoulli inequality yields
\begin{align*}
\limsup_{t\rightarrow\infty}\|u\|_{L^1(\Omega)}\le \frac{r}{\mu},
\end{align*}
and then concludes \eqref{lem211} with some $t_1>0$. \qquad$\Box$\medskip

Via the similar arguments like that in \cite{LK2025}, we will obtain the estimate of $\|u\|_{L^p(\Omega)}$ for some $p>\max\{\frac{n}{2},2\}$ with respect to $\mu$.

\begin{lemma}\label{lem22}
Let $\alpha\in(0,1)$ and $\mu>\mu_0$. Then there exists some $p_1>\max\{2,\frac{n}{2}\}$ and $t_2\ge t_1$ along with $L_4=L_4(\chi,\alpha,r,\Omega,p)>0$ such that for $p\in(1,p_1]$ it admits that
\begin{align}\label{lem221}
\|u\|_{L^p(\Omega)}\le \frac{L_4}{\mu},~~t>t_2.
\end{align}
\end{lemma}
{\bf Proof.}\
Multiply \eqref{p}$_1$ by $u^{p-1}$ with $p\ge 2$ and integrate over $\Omega$ by part to have
\begin{align}
\frac{1}{p}\frac{d}{dt}\int_\Omega u^pdx&=-(p-1)\int_\Omega u^{p-2}|\nabla u|^2dx+\chi(p-1)\int_\Omega u^{p-1}\frac{\nabla u\cdot\nabla v}{v^\alpha}dx\nonumber\\
&~~+r\int_\Omega u^pdx-\mu\int_\Omega u^{p+1}dx\nonumber\\
&\le -\frac{1}{p}\int_\Omega u^pdx-\frac{4(p-1)}{p^2}\int_\Omega |\nabla u^\frac{p}{2}|^2dx+\frac{\chi(p-1)\alpha}{p}\int_\Omega{u^{p}}\frac{|\nabla v|^2}{v^{1+\alpha}}dx\nonumber\\
&~~+\frac{\chi(p-1)}{p}\int_\Omega\frac{u^{p+1}}{v^\alpha}dx+(r+\frac{1}{p})\int_\Omega u^pdx-{\mu} \int_\Omega u^{p+1}dx,~~t>0.\label{222}
\end{align}
By the Young inequality with \eqref{lem132}, we know for $\alpha\in(0,1)$ that
\begin{align}\label{223}
\frac{\chi(p-1)}{p}\int_\Omega\frac{u^{p+1}}{v^\alpha}dx&\le \frac{p-1}{p^2}\epsilon_2\int_\Omega \frac{u^{p+1}}{v}dx+{C_4}{\epsilon_2^{-\frac{\alpha}{1-\alpha}}}\int_\Omega u^{p+1}dx\nonumber\\
&\le \frac{2(p-1)}{p^2}\epsilon_2\int_\Omega |\nabla u^\frac{p}{2}|^2dx+\frac{p-1}{p^2}\epsilon_2\int_\Omega u^pdx+C_4{\epsilon_2^{-\frac{\alpha}{1-\alpha}}}\int_\Omega u^{p+1}dx
\end{align}
with $\epsilon_2>0$ and $C_4=(\frac{\chi(p-1)}{p})^\frac{1}{1-\alpha}(\frac{p^2}{p-1})^\frac{\alpha}{1-\alpha}$. Again by the Young inequality with \eqref{lem132} and \eqref{lem133},
\begin{align}\label{224}
\frac{\chi(p-1)\alpha}{p}\int_\Omega{u^{p}}\frac{|\nabla v|^2}{v^{1+\alpha}}dx&\le \frac{p-1}{2p^2}\epsilon_3\int_\Omega u^p\frac{|\nabla v|^2}{v^2}dx+C_5\epsilon_3^{-\frac{\alpha}{1-\alpha}}\int_\Omega u^p\frac{|\nabla v|^2}{v}dx\nonumber\\
&\le \frac{2(p-1)}{p^2}\epsilon_3\int_\Omega|\nabla u^\frac{p}{2}|^2dx+\frac{p-1}{p^2}\epsilon_3\int_\Omega u^pdx\nonumber\\
&~~+C_5\epsilon_3^{-\frac{\alpha}{1-\alpha}}\big(\int_\Omega u^{p+1}dx\big)^\frac{p}{p+1}\big(\int_\Omega \frac{|\nabla v|^{2(p+1)}}{v^{p+1}}dx\big)^\frac{1}{p+1}\nonumber\\
&\le \frac{2(p-1)}{p^2}\epsilon_3\int_\Omega|\nabla u^\frac{p}{2}|^2dx+\frac{p-1}{p^2}\epsilon_3\int_\Omega u^pdx\nonumber\\
&~~+C_5\epsilon_3^{-\frac{\alpha}{1-\alpha}}\big(\int_\Omega u^{p+1}dx\big)^\frac{p}{p+1} \Big[(L_2+\epsilon_1)\int_\Omega u^{p+1}dx+\frac{L_3}{\epsilon_1^{p}}\big(\int_\Omega udx\big)^{p+1}\Big]^\frac{1}{p+1}\nonumber\\
&\le \frac{2(p-1)}{p^2}\int_\Omega|\nabla u^\frac{p}{2}|^2dx+\frac{p-1}{p^2}\int_\Omega u^pdx\nonumber\\
&~~+C_5\epsilon_3^{-\frac{\alpha}{1-\alpha}} \big((L_2+\epsilon_1)^\frac{1}{p+1}+\epsilon_1\big)\int_\Omega u^{p+1}dx\nonumber\\
&~~+ \frac{L_3C_5^{p+1}\epsilon_3^{-\frac{(p+1)\alpha}{1-\alpha}}}{\epsilon_1^{2p}}\big(\int_\Omega udx\big)^{p+1}
\end{align}
with $\epsilon_3>0$ and $C_5=(\frac{\chi\alpha (p-1)}{p})^\frac{1}{1-\alpha}(\frac{2p}{p-1})^\frac{\alpha}{1-\alpha}$. In addition, we have by the Young inequality that
\begin{align}\label{225}
(r+\frac{1}{p}+\frac{2(p-1)}{p^2})\int_\Omega u^pdx\le (r+2)\int_\Omega u^pdx\le \epsilon_1\mu\int_\Omega u^{p+1}dx+\frac{(r+2)^{p+1}|\Omega|}{(\epsilon_1\mu)^p}
\end{align}
A combination of \eqref{222}--\eqref{225} with $\epsilon_2=\epsilon_3=1$ entails that
\begin{align}\label{226}
\frac{1}{p}\frac{d}{dt}\int_\Omega u^pdx&\le -\frac{1}{p}\int_\Omega u^pdx+\Big(C_4+C_5 \big((L_2+\epsilon_1)^\frac{1}{p+1}+\epsilon_1\big)-(1-\epsilon_1)\mu\Big)\int_\Omega u^{p+1}dx\nonumber\\
&~~+\frac{L_3C_5^{p+1}}{\epsilon_1^{2p+1}}\big(\int_\Omega udx\big)^{p+1}+\frac{(r+2)^{p+1}|\Omega|}{(\epsilon_1\mu)^p},~~t>0.
\end{align}
If $\mu>\tilde\mu_0$, there exists some $p_1>\max\{2,\frac{n}{2}\big\}$ such that
$$C_4+C_5 L_2^\frac{1}{p_1+1}-\mu=\frac{\chi(p_1-1)}{p_1}(\chi p_1)^\frac{\alpha}{1-\alpha}+2(2\chi \alpha p_1)^\frac{1}{1-\alpha}(\frac{2p_1}{p_1-1})^\frac{1}{2}-\mu<0,$$
which yields that there exists some $\epsilon_1>0$ such that $C_4+C_5 \big((L_2+\epsilon_1)^\frac{1}{p_1+1}+\epsilon_1\big)-(1-\epsilon_1)\mu<0$. Hence, for $\mu>\mu_0$, we obtain by \eqref{lem211} that
\begin{align*}
\frac{1}{p_1}\frac{d}{dt}\int_\Omega u^{p_1}dx&\le -\frac{1}{p_1}\int_\Omega u^{p_1}dx+\frac{C_6}{\mu^{p_1+1}}+\frac{C_6}{\mu^{p_1}}\le -\frac{1}{p_1}\int_\Omega u^{p_1}dx+\frac{2C_6}{\mu^{p_1}},~~t>t_1
\end{align*}
with some $C_6=C_6(\chi,\alpha,r,\Omega,p_1)>0$, which along with the Bernoulli inequality entails that
\begin{align*}
\limsup_{t\rightarrow\infty}\|u\|_{L^{p_1}(\Omega)}\le \frac{(2C_6p_1)^\frac{1}{p_1}}{\mu}.
\end{align*}
This yields that there exists some $t_2\ge t_1$ such that
\begin{align*}
\|u\|_{L^{p_1}(\Omega)}\le \frac{2(2C_6p_1)^\frac{1}{p_1}}{\mu},~~t>t_2,
\end{align*}
and then concludes \eqref{lem221} by the Young inequality with $L_4=2(2C_6p_1)^\frac{1}{p_1}|\Omega|^\frac{p_1-p}{pp_1}$. \qquad$\Box$\medskip

Here, we deal with $\|u\|_{L^p(\Omega)}$ for $p>1$  with respect to $\mu$.
\begin{lemma}\label{lem23}
Let $\alpha\in(0,1)$ and $\mu>\mu_0$. Then for $p>1$ there exists some $t_3\ge t_2$ along with some $L_5=L_5(\chi,\alpha,r,\Omega,p)>0$ such that
\begin{align}\label{lem231}
\| u\|_{L^p(\Omega)}\le \frac{L_5}{\mu},~~t>t_3.
\end{align}
\end{lemma}
{\bf Proof.}\
If $\alpha\in(0,1)$ and $\mu>\mu_0$ with $p\ge 2$, we know by \eqref{222}, \eqref{223} and \eqref{224} with $\epsilon_2=\epsilon_3=\frac{1}{2}$, and \eqref{lem133} with $\epsilon_1=1$ that
\begin{align}\label{231}
\frac{1}{p}\frac{d}{dt}\int_\Omega u^pdx&\le -\frac{1}{p}\int_\Omega u^pdx-\frac{4(p-1)}{p^2}\int_\Omega |\nabla u^\frac{p}{2}|^2dx+\frac{\chi(p-1)\alpha}{p}\int_\Omega{u^{p}}\frac{|\nabla v|^2}{v^{1+\alpha}}dx\nonumber\\
&~~+\frac{\chi(p-1)}{p}\int_\Omega\frac{u^{p+1}}{v^\alpha}dx+(r+\frac{1}{p})\int_\Omega u^pdx\nonumber\\
&\le -\frac{1}{p}\int_\Omega u^pdx-\frac{2(p-1)}{p^2}\int_\Omega |\nabla u^\frac{p}{2}|^2dx+C_7\int_\Omega u^{p+1}dx\nonumber\\
&~~+C_8\int_\Omega u^pdx+C_9\big(\int_\Omega udx\big)^{p+1},~~t>t_2
\end{align}
with $C_7=2^\frac{\alpha}{1-\alpha}\Big(C_4+C_5 \big((L_2+1)^\frac{1}{p+1}+1\big)\Big), C_8=r+\frac{2p-1}{p^2}, C_9=L_3C_5^{p+1}2^\frac{(p+1)\alpha}{1-\alpha}$. Applying the Gaglirado-Nirenberg inequality and the Poincar\'{e} inequality with \eqref{lem221} for some $p_1>\max\{2,\frac{n}{2}\}$, we know for $p>p_1$ that
\begin{align}\label{232}
\int_\Omega u^{p+1}&=\|u^\frac{p}{2}\|_{L^\frac{2(p+1)}{p}(\Omega)}^\frac{2(p+1)}{p}\nonumber\\
&\le C_{GN}\|\nabla u^\frac{p}{2}\|_{L^2(\Omega)}^{\frac{2(p+1)}{p}a}\|u^\frac{p}{2}\|_{L^\frac{2p_1}{p}(\Omega)}^{\frac{2(p+1)}{p}(1-a)}
+C_{GN}\|u^\frac{p}{2}\|_{L^\frac{2p_1}{p}(\Omega)}^\frac{2(p+1)}{p}\nonumber\\
&\le \frac{p-1}{C_7p^2}\|\nabla u^\frac{p}{2}\|_{L^2(\Omega)}^2+C_{GN}\|u\|_{L^{p_1}(\Omega)}^{p+1}
+C_{10}\|u\|_{L^{p_1}(\Omega)}^{\frac{(p+1)(1-a)}{p-(p+1)a}p},
\end{align}
where $a=\frac{\frac{pn}{2p_1}-\frac{pn}{2(p+1)}}{1-\frac{n}{2}+\frac{pn}{2p_1}}\in(0,1)$ and $\frac{p+1}{p}a<1$ with some $C_{10}=C_{GN}^\frac{p}{p-(p+1)a}(\frac{p-1}{C_7p^2})^{-\frac{(p+1)a}{p-(p+1)a}}$, and
\begin{align}\label{233}
\int_\Omega u^{p}&=\|u^\frac{p}{2}\|_{L^2(\Omega)}^2\nonumber\\
&\le C_{GN}\|\nabla u^\frac{p}{2}\|_{L^2(\Omega)}^{2b}\|u^\frac{p}{2}\|_{L^\frac{2p_1}{p}(\Omega)}^{2(1-b)}
+C_{GN}\|u^\frac{p}{2}\|_{L^\frac{2p_1}{p}(\Omega)}^2\nonumber\\
&\le \frac{p-1}{C_8p^2}\|\nabla u^\frac{p}{2}\|_{L^2(\Omega)}^2+C_{11}\|u\|_{L^{p_1}(\Omega)}^{p}
,
\end{align}
where $b=\frac{\frac{pn}{2p_1}-\frac{n}{2}}{1-\frac{n}{2}+\frac{pn}{2p_1}}\in(0,1)$ with some $C_{11}=C_{GN}+C_{GN}^\frac{1}{1-b}(\frac{p-1}{C_8p^2})^{-\frac{b}{1-b}}$. Hence, if $\alpha\in(0,1)$ and $\mu>\mu_0$, a combination of \eqref{231}--\eqref{233} with \eqref{lem211} and \eqref{lem221} shows for $p>p_1$ that
\begin{align*}
\frac{1}{p}\frac{d}{dt}\int_\Omega u^pdx&\le-\frac{1}{p}\int_\Omega u^pdx+C_9\| u\|_{L^1(\Omega)}^{p+1}+C_{GN}\|u\|_{L^{p_1}(\Omega)}^{p+1}
+C_{10}\|u\|_{L^{p_1}(\Omega)}^{\frac{1-a}{1-\frac{p+1}{p}a}(p+1)}+C_{11}\|u\|_{L^{p_1}(\Omega)}^{p}\nonumber\\
&\le -\frac{1}{p}\int_\Omega u^pdx+\frac{C_9(r+1)^{p+1}}{\mu^{p+1}}
+\frac{C_{GN}L_4^{p+1}}{\mu^{p+1}}+\frac{C_{10}L_4^{\frac{(p+1)(1-a)}{p-(p+1)a}p}}
{\mu^{\frac{(p+1)(1-a)}{p-(p+1)a}p}}+\frac{C_{11}L_4^{p}}{\mu^{p}}\nonumber\\
&\le -\frac{1}{p}\int_\Omega u^pdx+\frac{C_{12}}{\mu^{p}},~~t>t_2
\end{align*}
with $C_{12}=C_9(2r)^{p+1}+C_{GN}L_4^{p+1}+C_{10}L_4^{\frac{(p+1)(1-a)}{p-(p+1)a}p}+C_{11}L_4^{p}$, which together with the Bernoulli inequality entails that
\begin{align*}
\limsup_{t\rightarrow\infty}\|u\|_{L^p(\Omega)}\le \frac{(C_{12}p)^\frac{1}{p}}{\mu}.
\end{align*}
This yields that there exists some $t_3\ge t_2$ such that
\begin{align*}
\|u\|_{L^p(\Omega)}\le \frac{2(C_{12}p)^\frac{1}{p}}{\mu},~~t>t_3,
\end{align*}
and then concludes \eqref{lem231} with $L_5=2(C_{12}p)^\frac{1}{p}$ by the Young inequality. \qquad$\Box$\medskip

{\bf Proof of Theorem \ref{th1}}\
 According to the representation of $u$ as follows
\begin{align}\label{241}
u(x,t)&={\rm e}^{(t-t_3-1)(\Delta-1)}u(x,t_3+1)-\chi\int_{t_3+1}^{t}{\rm e}^{(t-s)(\Delta-1)}\nabla\cdot(u\frac{\nabla v}{v^\alpha})ds\nonumber\\
&~~+(1+r) \int_{t_3+1}^t{\rm e}^{(t-s)(\Delta-I)}uds-\mu\int_{t_3+1}^t{\rm e}^{(t-s)(\Delta-1)} u^2 ds,~~x\in\Omega,~t>t_3+2,
\end{align}
it is known from \cite[Lemma 1.3]{W2010} that
\begin{align}\label{242}
\|u\|_{L^{\infty}(\Omega)}&\le\|{\rm e}^{(t-t_3-1)(\Delta-1)}u(\cdot,t_3+1)\|_{L^{\infty}(\Omega)}+\chi\int_{t_3+1}^t\|{\rm e}^{(t-s)(\Delta-1)}\nabla\cdot(u\frac{\nabla v}{v^\alpha})\|_{L^{\infty}(\Omega)}ds\nonumber\\
&~~+(1+r)\int_{t_3+1}^t\|{\rm e}^{(t-s)(\Delta-1)}u\|_{L^{\infty}(\Omega)}ds\nonumber\\
&\le C_{13}\big(1+(t-t_3-1)^{-\frac{n}{2(n+1)}}\big){\rm e}^{-\lambda_1(t-t_3-1)}\|u(\cdot,t_3+1)-\overline{u}(\cdot,t_3+1)\|_{L^{n+1}(\Omega)}\nonumber\\
&~~+\|{\rm e}^{(t-t_3-1)(\Delta-1)}\bar u(\cdot,t_3+1)\|_{L^{\infty}(\Omega)}+{\chi C_{13}}\int_{t_3+1}^t(1+(t-s)^{-\frac{q+n}{2q}}){\rm e}^{-\lambda_1(t-s)}\|u\frac{\nabla v}{v^\alpha}\|_{L^q(\Omega)}ds\nonumber\\
&~~+(1+r)C_{13}\int_{t_3+1}^t\big(1+(t-s)^{-\frac{n}{2(n+1)}}\big){\rm e}^{-\lambda_1(t-s)}\|u-\overline{u}\|_{L^{n+1}(\Omega)}ds\nonumber\\
&~~+(1+r)\int_{t_3+1}^t\|{\rm e}^{(t-s)(\Delta-1)} \bar u\|_{L^{\infty}(\Omega)}ds\nonumber\\
 &\le 4C_{13}\|u(\cdot, t_3+1)\|_{L^{n+1}(\Omega)}+\frac{1}{|\Omega|}\|u(\cdot,t_3+1)\|_{L^1(\Omega)}+C_{14}\sup_{t>{t_3+2}}\|u\frac{\nabla v}{v^\alpha}\|_{L^{q}(\Omega)}\nonumber\\
 &~~+C_{15}\sup_{t>{t_3+2}}\|u\|_{L^{n+1}(\Omega)}+\frac{1+r}{|\Omega|}\sup_{t>{t_3+2}}\|u\|_{L^{1}(\Omega)},~~t>t_3+2,
\end{align}
where $\bar u=\frac{1}{|\Omega|}\int_\Omega udx$ and $\lambda_1>0$ is the first nonzero
eigenvalue of $-\Delta $ in $\Omega$ under Neumann boundary conditions, and $C_{13}=C_{13}(\Omega)>0$ along with $C_{14}={\chi C_{13}}\int_0^\infty(1+\sigma^{-\frac{q+n}{2q}}){\rm e}^{-\lambda_1\sigma}d\sigma$ with some $q>n$ determined below, and $C_{15}=2(1+r)C_{13}\int_0^\infty(1+\sigma^{-\frac{n}{2(n+1)}}){\rm e}^{-\lambda_1\sigma}d\sigma$.

If $\alpha\in(0,\frac{1}{2}]$ and $\mu>\mu_0$, letting $q=n+1$, we get by the H\"{o}lder inequality with \eqref{131} for $k=p$ that
\begin{align}\label{243}
\int_\Omega (u\frac{|\nabla v|}{v^\alpha})^{n+1}dx&\le \Big(\int_\Omega u^{(2-\alpha)(n+1)}dx\Big)^\frac{1}{2-\alpha}\Big(\int_\Omega v^{(2-\alpha)(n+1)}dx\Big)^\frac{1-2\alpha}{2(2-\alpha)}\Big(\int_\Omega \big(\frac{|\nabla v|^2}{v}\big)^{(2-\alpha)(n+1)}dx\Big)^\frac{1}{2(2-\alpha)}\nonumber\\
&\le \Big(\int_\Omega u^{(2-\alpha)(n+1)}dx\Big)^\frac{3-2\alpha}{2-\alpha}\Big(C_1\int_\Omega u^{(2-\alpha)(n+1)}dx+C_2\int_\Omega v^{(2-\alpha)(n+1)}dx\Big)^\frac{1}{2(2-\alpha)} \nonumber\\
&\le C_{16}\|u\|_{L^{(2-\alpha)(n+1)}(\Omega)}^{(2-\alpha)(n+1)}
\end{align}
with $C_{16}=2+[C_1\big((2-\alpha)(n+1)\big)+C_2\big((2-\alpha)(n+1),\Omega\big)]^\frac{1}{2(2-\alpha)}$ due to $\int_\Omega v^{(2-\alpha)(n+1)}dx\le \int_\Omega u^{(2-\alpha)(n+1)}dx$, and then
by \eqref{242} with \eqref{243}, \eqref{lem211} and \eqref{lem231} that
\begin{align}\label{244}
\|u\|_{L^\infty(\Omega)}\le \frac{C_{17}}{\mu}+\frac{C_{18}}{\mu^{2-\alpha}}\le\frac{C_{19}}{\mu},~~t>t_3+2,
\end{align}
with $C_{17}=(4C_{13}+C_{15})L_5(\chi,\alpha,r,\Omega,n+1)+\frac{2r(r+2)}{|\Omega|}$, $C_{18}=C_{14}C_{17}^\frac{1}{n}L_5(\chi,\alpha,r,\Omega,(2-\alpha)(n+1))$, and $C_{19}=C_{17}+C_{18}$.

If $\alpha\in(\frac{1}{2},\frac{n+2}{2n})$ and $\mu>\mu_0$, with selecting $q=\frac{2\alpha}{2\alpha-1}>n$, again by the H\"{o}lder inequality with \eqref{131} for $k=p$ and \eqref{lem131}, we obtain that
 \begin{align}\label{2431}
\int_\Omega (u\frac{|\nabla v|}{v^\alpha})^\frac{2\alpha}{2\alpha-1}dx&\le \Big(\int_\Omega u^\frac{2\alpha(2-\alpha)}{(1-\alpha)(2\alpha-1)}dx\Big)^\frac{(3-2\alpha)(1-\alpha)}{2(2-\alpha)}\Big(\int_\Omega \big(\frac{|\nabla v|^2}{v}\big)^{\frac{2\alpha(2-\alpha)}{(1-\alpha)(2\alpha-1)}}dx\Big)^\frac{1-\alpha}{2(2-\alpha)}\Big(\int_\Omega \frac{u}{v}dx\Big)^\alpha\nonumber\\
&\le |\Omega|^\alpha \Big(\int_\Omega u^\frac{2\alpha(2-\alpha)}{(1-\alpha)(2\alpha-1)}dx\Big)^\frac{(3-2\alpha)(1-\alpha)}{2(2-\alpha)}\Big(C_1\int_\Omega u^{\frac{2\alpha(2-\alpha)}{(1-\alpha)(2\alpha-1)}}dx+C_2\int_\Omega v^{\frac{2\alpha(2-\alpha)}{(1-\alpha)(2\alpha-1)}}dx\Big)^\frac{1-\alpha}{2(2-\alpha)}\nonumber\\
&\le C_{20}\| u\|_{L^{\frac{2\alpha(2-\alpha)}{(1-\alpha)(2\alpha-1)}}(\Omega)}^\frac{2\alpha(2-\alpha)}{2\alpha-1}
\end{align}
with $C_{20}=|\Omega|^\alpha\big[C_1\big(\frac{2\alpha(2-\alpha)}{(1-\alpha)(2\alpha-1)}\big)
+C_2\big(\frac{2\alpha(2-\alpha)}{(1-\alpha)(2\alpha-1)},\Omega\big)\big]^\frac{1}{2(2-\alpha)}$ due to $\int_\Omega v^{\frac{2\alpha(2-\alpha)}{(1-\alpha)(2\alpha-1)}}dx\le \int_\Omega u^{\frac{2\alpha(2-\alpha)}{(1-\alpha)(2\alpha-1)}}dx$. Hence, it is known by \eqref{242} with \eqref{2431}, \eqref{lem211} and \eqref{lem231} that
\begin{align}\label{244}
\|u\|_{L^\infty(\Omega)}\le \frac{C_{21}}{\mu}+\frac{C_{22}}{\mu^{2-\alpha}}\le\frac{C_{23}}{\mu},~~t>t_3+2,
\end{align}
with $C_{21}=(4C_{13}+C_{15})L_5(\chi,\alpha,r,\Omega,n+1)+\frac{2r(r+2)}{|\Omega|}$, $C_{22}=C_{14}C_{17}^\frac{1}{n}L_5(\chi,\alpha,r,\Omega,\frac{2\alpha(2-\alpha)}{(1-\alpha)(2\alpha-1)})$, and $C_{23}=C_{21}+C_{22}$.

Consequently, if $\alpha\in(0,\frac{n+2}{2n})$ and $\mu>\mu_0$, we obtain \eqref{th11} with $L_0=\max\{C_{20},C_{23}\}$ and $t_0=t_3+2$.
\qquad$\Box$\medskip

\section{H\"{o}lder regularity for $\alpha\in(0,\frac{1}{2})$}
Let $(u,v)$ be the global bounded classical solution to \eqref{p} mentioned in Lemma \ref{lem11}. Denote $z:=v^{1-\alpha}$ for $\alpha\in(0,1)$. Then
\begin{equation}\label{pp}
\begin{cases}
u_t=\Delta u-\frac{\chi}{1-\alpha}\nabla\cdot(u\nabla z)+ru-\mu u^2,&~~x\in\Omega,~t>0,\\
0=\Delta z+\frac{\alpha}{1-\alpha}\frac{|\nabla z|^2}{z}-(1-\alpha) z+(1-\alpha) u{z^{-\frac{\alpha}{1-\alpha}}},&~~x\in\Omega,~t>0,\\
\frac{\partial u}{\partial\nu}=\frac{\partial z}{\partial\nu}=0,&~~x\in\partial\Omega,~t>0,\\
u(x,0)=u_0(x),&~~x\in\Omega.
\end{cases}
\end{equation}
We will establish the H\"{o}lder regularity of $u$ and estimate of $\nabla u$ in $L^p$-norm to obtain the convergence of $(u,v)$.  At first, we estimate $\int_\Omega |\nabla z|^{2q}dx$ for $q\ge 3$.

\begin{lemma}\label{lem41}
If $\alpha\in(0,\frac{1}{2})$ and $\mu>\tilde\mu_0$, it holds for $q\ge 3$ that
\begin{align}\label{lem411}
\int_\Omega |\nabla z|^{2q}dx&\le-\frac{1}{q}\int_\Omega |\nabla|\nabla z|^q|^2dx+4(1-\alpha)^2(q+n)\int_\Omega u^2z^{-\frac{2\alpha}{1-\alpha}}|\nabla z|^{2(q-1)}dx.
\end{align}
\end{lemma}
{\bf Proof.}\
Based on the equality $\Delta|\nabla z|^2=2\nabla z\cdot\nabla\Delta z+2|D^2z|^2$, testing $z$-equation by $2\nabla z\cdot\nabla$ to get
\begin{align}\label{4110}
0&=2\nabla z\cdot\nabla(\Delta z+\frac{\alpha}{1-\alpha}\frac{|\nabla z|^2}{z}-(1-\alpha) z+(1-\alpha) uz^{-\frac{\alpha}{1-\alpha}})\nonumber\\
&=\Delta|\nabla z|^2-2|D^2z|^2+\frac{2\alpha}{1-\alpha}\nabla z\cdot\frac{\nabla|\nabla z|^2}{z}\nonumber\\
&~~-\frac{2\alpha}{1-\alpha}\frac{|\nabla z|^4}{z^2}-2(1-\alpha)|\nabla z|^2+2(1-\alpha)\nabla z\cdot\nabla (uz^{-\frac{\alpha}{1-\alpha}}),
\end{align}
which multiplying $|\nabla z|^{2(q-1)}$ with $q\ge 2$ and integrating over $\Omega$ by part along with the convexity of $\Omega$ yields
\begin{align}\label{411}
0&=\int_\Omega \Delta|\nabla z|^2|\nabla z|^{2(q-1)}dx-2\int_\Omega |\nabla z|^{2(q-1)}|D^2z|^2dx\nonumber\\
&~~-2(1-\alpha)\int_\Omega |\nabla z|^{2q}dx-\frac{2\alpha}{1-\alpha}\int_\Omega \frac{|\nabla z|^{2(q+1)}}{z^2}dx\nonumber\\
&~~+\frac{2\alpha}{1-\alpha}\int_\Omega |\nabla z|^{2(q-1)}\nabla z\cdot\frac{\nabla|\nabla z|^2}{z}dx+2(1-\alpha)\int_\Omega |\nabla z|^{2(q-1)}\nabla z\cdot\nabla(uz^{-\frac{\alpha}{1-\alpha}})dx\nonumber\\
&\le-\frac{4}{q^2}(q-1-\frac{\alpha}{2(1-\alpha)})\int_\Omega |\nabla|\nabla z|^q|^2dx-2(1-\alpha)\int_\Omega |\nabla z|^{2q}dx\nonumber\\
&~~-2\int_\Omega |\nabla z|^{2(q-1)}|D^2z|^2dx+2(1-\alpha)\int_\Omega |\nabla z|^{2(q-1)}\nabla z\cdot(uz^{-\frac{\alpha}{1-\alpha}})dx.
\end{align}
By the pointwise inequality $|\Delta z|^2\le n|D^2z|^2$, the last term of \eqref{411} can be estimated as
\begin{align}\label{414}
2(1-\alpha)\int_\Omega |\nabla z|^{2(q-1)}\nabla z&\cdot\nabla(uz^{-\frac{\alpha}{1-\alpha}})dx=-2(1-\alpha)\int_\Omega uz^{-\frac{\alpha}{1-\alpha}}|\nabla z|^{2(q-1)}\Delta zdx\nonumber\\
&~~-2(1-\alpha)(q-1)\int_\Omega uz^{-\frac{\alpha}{1-\alpha}}|\nabla z|^{2(q-2)}\nabla|\nabla z|^2\cdot\nabla zdx\nonumber\\
&\le 2\int_\Omega |\nabla z|^{2(q-1)}|D^2 z|^2dx+\frac{q-1}{q^2}\int_\Omega |\nabla|\nabla z|^q|^2dx\nonumber\\
&~~+(1-\alpha)^2(4q-4+\frac{n}{2})\int_\Omega u^2z^{-\frac{2\alpha}{1-\alpha}}|\nabla z|^{2(q-1)}.
\end{align}
Hence, if $\alpha\in(0,\frac{1}{2})$ with $\mu>\tilde\mu_0$, then for $q\ge 3$ we get by \eqref{411} and \eqref{414} that
\begin{align*}
\int_\Omega |\nabla z|^{2q}dx&\le-\frac{2}{q^2}(q-1-\frac{\alpha}{1-\alpha})\int_\Omega |\nabla|\nabla z|^q|^2dx\nonumber\\
&~~+(1-\alpha)^2(4q-4+n)\int_\Omega u^2z^{-\frac{2\alpha}{1-\alpha}}|\nabla z|^{2(q-1)}\nonumber\\
&\le -\frac{1}{q}\int_\Omega |\nabla|\nabla z|^q|^2dx+4(1-\alpha)^2(q+n)\int_\Omega u^2z^{-\frac{2\alpha}{1-\alpha}}|\nabla z|^{2(q-1)}dx.
\end{align*}
This completes the proof of \eqref{lem411}. \qquad$\Box$\medskip

For convenience, let $\theta:=\frac{2\alpha}{1-2\alpha}$ for $\alpha\in(0,\frac{1}{2})$, and $\theta_0:=\min\{\theta,1\}$. Then the following recursive relation on $\int_\Omega |\nabla z|^{2q}dx+1$ can be obtained.
\begin{lemma}\label{lem42}
If $\alpha\in(0,\frac{1}{2})$ and $\mu>\tilde\mu_0$, then for $q>\max\{1+\theta,n\}$ it admits that
\begin{align}\label{lem421}
\int_\Omega |\nabla z|^{2q}dx+1&\le
 C_{30}q^\frac{2n+2+\theta_0}{2-\theta_0}\Big(\int_\Omega |\nabla z|^{q}dx+1\Big)^2
\end{align}
with some $C_{30}>0$ independent of $q$.
\end{lemma}
{\bf Proof.}\
For $\alpha\in(0,\frac{1}{2})$ and $\mu>\tilde\mu_0$, if $q>1+\theta$, we obtain by \eqref{lem111} and \cite[Lemma 3.8]{K2026} along with the definition of $z$ that
\begin{align}\label{421}
\int_\Omega u^2z^{-\frac{2\alpha}{1-\alpha}}|\nabla z|^{2(q-1)}dx&=(1-\alpha)^{2\theta}\int_\Omega u^2\frac{|\nabla v|^{2\theta}}{v^\theta}|\nabla z|^{2(q-1-\theta)}dx\nonumber\\
&\le \|u\|_{L^\infty(\Omega)}^2 \Big(\int_\Omega \big(\frac{|\nabla v|^2}{v}\big)^\frac{\theta n}{\theta_0}dx\Big)^\frac{\theta_0}{n}\Big(\int_\Omega |\nabla z|^\frac{2(q-1-\theta)n}{n-\theta_0}dx\Big)^\frac{n-\theta_0}{n}\nonumber\\
&\le \|u\|_{L^\infty(\Omega)}^2 \Big(C_{24}\int_\Omega u^\frac{\theta n}{\theta_0}dx\Big)^\frac{\theta_0}{n}\Big(\int_\Omega |\nabla z|^\frac{2(q-1-\theta)n}{n-\theta_0}dx\Big)^\frac{n-\theta_0}{n}\nonumber\\
&\le (C_{24}|\Omega|)^\frac{\theta_0}{n}\|u\|_{L^\infty(\Omega)}^{2+\theta}\Big(\int_\Omega |\nabla z|^\frac{2(q-1-\theta)n}{n-\theta_0}dx\Big)^\frac{n-\theta_0}{n}\nonumber\\
&\le C_{24}\Big(\int_\Omega |\nabla z|^\frac{2(q-1-\theta)n}{n-\theta_0}dx\Big)^\frac{n-\theta_0}{n}
\end{align}
with some $C_{24}=C_{24}(\Omega)>0$ and $C_{25}=(C_{24}|\Omega|)^\frac{\theta_0}{n} L_1^{2+\theta}$. By the Gaglirado-Nirenberg inequality and the Poincar\'{e} inequality, it holds that
\begin{align}\label{422}
\Big(\int_\Omega |\nabla z|^\frac{2(q-1-\theta)n}{n-\theta_0}dx\Big)^\frac{n-\theta_0}{n}&\le|\Omega|^\frac{(1+\theta)(n-\theta_0)}{qn}
\Big(\int_\Omega |\nabla z|^\frac{2nq}{n-\theta_0}dx\Big)^\frac{(n-\theta_0)(q-1-\theta)}{nq}\nonumber\\
&\le C_{26}\||\nabla z|^q\|_{L^\frac{2n}{n-\theta_0}(\Omega)}^\frac{2(q-1-\theta)}{q}\nonumber\\
&\le C_{26} \big(C_{27}\||\nabla z|^q\|_{W^{1,2}(\Omega)}^a\||\nabla z|^q\|_{L^1(\Omega)}^{1-a}\big)^\frac{2(q-1-\theta)}{q}\nonumber\\
&\le C_{28}\Big(\|\nabla|\nabla z|^q\|_{L^2(\Omega)}^\frac{2(q-1-\theta)a}{q}\||\nabla z|^q\|_{L^1(\Omega)}^\frac{2(q-1-\theta)(1-a)}{q}+\||\nabla z|^q\|_{L^1(\Omega)}^\frac{2(q-1-\theta)}{q}\Big)
\end{align}
where $a=\frac{n+\theta_0}{n+2}\in(0,1)$ due  to $\theta_0\in(0,1]$, with $C_{26}=(1+|\Omega|)^\frac{(1+\theta)(n-\theta_0)}{n}$ and $C_{27}=C_{27}(\Omega)>0$ and $C_{28}=2C_{26}(1+C_{27})^2$. Then, we get from by \eqref{421} with \eqref{422} and the Young inequality with $\epsilon_2>0$ that
\begin{align}\label{423}
\int_\Omega u^2z^{-\frac{2\alpha}{1-\alpha}}|\nabla z|^{2(q-1)}dx&\le  C_{29}\Big(\|\nabla|\nabla z|^q\|_{L^2(\Omega)}^\frac{2(q-1-\theta)a}{q}\||\nabla z|^q\|_{L^1(\Omega)}^\frac{2(q-1-\theta)(1-a)}{q}+\||\nabla z|^q\|_{L^1(\Omega)}^\frac{2(q-1-\theta)}{q}\Big)\nonumber\\
&\le \epsilon_2\|\nabla|\nabla z|^q\|_{L^2(\Omega)}^2+C_{29}\||\nabla z|^q\|_{L^1(\Omega)}^\frac{2(q-1-\theta)}{q}\nonumber\\
&~~+\epsilon_2^{-\frac{(q-1-\theta)a}{q-(q-1-\theta)a}}C_{29}^{\frac{q}{q-(q-1-\theta)a}}\||\nabla z|^q\|_{L^1(\Omega)}^{\frac{2(q-1-\theta)(1-a)}{q-(q-1-\theta)a}}
\end{align}
with $C_{29}=C_{24}C_{28}$.  If $q>1+\theta$,  it is known that
$$\frac{(q-1-\theta)a}{q-(q-1-\theta)a}<\frac{n+\theta_0}{2-\theta_0},~~\frac{q}{q-(q-1-\theta)a}<\frac{n+2}{2-\theta_0}$$
and
$$\frac{2(q-1-\theta)(1-a)}{q-(q-1-\theta)a}<2.$$
Hence, if $\alpha\in(0,\frac{1}{2})$ and $\mu>\tilde\mu_0$, for $q>\max\{n,1+\theta\}$ with selecting $\epsilon_2=\frac{1}{4(1-\alpha)^2q(q+n)}$,
we obtain by \eqref{lem411} and \eqref{423} that
\begin{align}\label{425}
\int_\Omega |\nabla z|^{2q}dx+1&\le 4(1-\alpha)^2(q+n)C_{29}\||\nabla z|^q\|_{L^1(\Omega)}^\frac{2(q-1-\theta)}{q}\nonumber\\
&~~+4(1-\alpha)^2(q+n)\big(4(1-\alpha)^2q(q+n)\big)^{\frac{(q-1-\theta)a}{q-(q-1-\theta)a}}
C_{29}^{\frac{q}{q-(q-1-\theta)a}}\||\nabla z|^q\|_{L^1(\Omega)}^{\frac{2(q-1-\theta)(1-a)}{q-(q-1-\theta)a}}+1\nonumber\\
&\le 8(1-\alpha)^2C_{29}q\||\nabla z|^q\|_{L^1(\Omega)}^\frac{2(q-1-\theta)}{q}\nonumber\\
&~~+(8(1-\alpha)^2)^{\frac{q}{q-(q-1-\theta)a}}C_{29}^{\frac{q}{q-(q-1-\theta)a}}
q^{1+\frac{2(q-1-\theta)a}{q-(q-1-\theta)a}}\||\nabla z|^q\|_{L^1(\Omega)}^{\frac{2(q-1-\theta)(1-a)}{q-(q-1-\theta)a}}+1\nonumber\\
&\le  C_{30}q^\frac{2n+2+\theta_0}{2-\theta_0}\Big(\int_\Omega |\nabla z|^{q}dx+1\Big)^2
\end{align}
with $C_{30}=8(1-\alpha)^2C_{29}+
(8(1-\alpha)^2(1+C_{29}))^\frac{n+2}{2-\theta_0}$ independent of $q$.
This completes the proof of \eqref{lem421}. \qquad$\Box$\medskip

According to Lemma \ref{lem42}, the boundedness of $\nabla z$ in $L^\infty$ for $\alpha\in(0,\frac{1}{2})$ can be obtained via the Moser iteration.
\begin{lemma}\label{lem43}
If $\alpha\in(0,\frac{1}{2})$ and $\mu>\tilde\mu_0$, there exist some $L_6=L_6(\chi,\alpha,r,\mu,\Omega)>0$ such that
\begin{align}\label{lem431}
\|\nabla z\|_{L^\infty(\Omega)}\le L_6.
\end{align}
\end{lemma}
{\bf Proof.}\
 If $\alpha\in(0,\frac{1}{2})$ and $\mu>\tilde\mu_0$, based on the elliptic regularity theory with the boundedness of $u$ in \eqref{lem111}, there exists some $C_{31}=C_{31}(\chi,\alpha,r,\mu,\Omega)>0$ such that
\begin{align}\label{4260}
\|\nabla v\|_{L^\infty(\Omega)}\le C_{31}.
\end{align}
Denote
$k_0:=\inf\big\{{k\in\mathbb{N}}|2^k\ge \max\{n,1+\theta,\frac{3}{\alpha}\}\big\}$, we get by \eqref{131} for $k=p$ and \eqref{4260} that
\begin{align}\label{426}
\Big(\int_\Omega |\nabla z|^{2^{k_0}}dx+1\Big)^\frac{1}{2^{k_0}}&\le(1-\alpha)\Big(\int_\Omega \frac{|\nabla v|^{2^{k_0}}}{v^{\alpha{2^{k_0}}}}dx\Big)^\frac{1}{2^{k_0}}+1\nonumber\\
&\le \|\nabla v\|_{L^\infty(\Omega)}^{1-2\alpha}\Big(\int_\Omega \big(\frac{|\nabla v|^2}{v}\big)^{\alpha{2^{k_0}}}dx\Big)^\frac{1}{2^{k_0}}+1\nonumber\\
&\le \|\nabla v\|_{L^\infty(\Omega)}^{1-2\alpha}\Big(C_1\int_\Omega u^{\alpha{2^{k_0}}}dx+C_2\int_\Omega v^{\alpha{2^{k_0}}}dx\Big)^\frac{1}{2^{k_0}}+1\nonumber\\
&\le C_{32}
\end{align}
with $C_{32}=1+|\Omega|^\frac{1}{{2^{k_0}}}
C_{31}^{1-2\alpha}L_1^{\alpha}\big(C_1(\alpha{2^{k_0}})+C_2(\alpha{2^{k_0}},\Omega)\big)^\frac{1}{2^{k_0}}$ due to $\int_\Omega v^{\alpha{2^{k_0}}}dx\le \int_\Omega u^{\alpha{2^{k_0}}}dx$. Now, let $q_k:=2^k$ for $k=k_0,k_0+1,\cdots,$
and
\begin{align*}
M(2^k)=\Big(\int_\Omega |\nabla z|^{2^k}dx+1\Big)^\frac{1}{2^k}, ~~k=k_0,k_0+1,\cdots,
\end{align*}
then it is obtained by \eqref{lem421} and \eqref{426} that
\begin{align*}
M(2^{k+1})&\le C_{46}^\frac{1}{2^{k+1}}2^\frac{(2n+2+\theta_0)k}{(2-\theta_0)2^{k+1}}M(2^k)\nonumber\\
&\le C_{46}^{{\sum_{i={k_0}}^{k}}\frac{1}{2^{i+1}}}\cdot2^{\frac{2n+6-\theta_0}{2-\theta_0}
\sum_{i=k_0}^{k}\frac{i}{2^{i+1}}}\cdot M(2^{k_0})\nonumber\\
&=C_{32}C_{46}^{\frac{1}{2^{k_0}}
}2^{\frac{(2n+2+\theta_0)(k_0+1)}{(2-\theta_0) 2^{k_0}}}.
\end{align*}
Letting $k\rightarrow\infty$, then
\begin{align*}
\|\nabla z\|_{L^\infty(\Omega)}\le  L_6
\end{align*}
with $ L_6=C_{32}C_{46}^{\frac{1}{2^{k_0}}}2^{\frac{(2n+2+\theta_0)(k_0+1)}{(2-\theta_0)2^{k_0}}},$
which concludes \eqref{lem431}. \qquad$\Box$\medskip

Based on the boundedness of $u,\nabla z$, we will get the desired H\"{o}lder regularity of $(u,v)$.

\begin{lemma}\label{lem44}
If $\alpha\in(0,\frac{1}{2})$ and $\mu>\tilde\mu_0$, there exists some $\delta_0\in(0,1)$ and $L_7=L_7(\chi,\alpha,r,\mu,\Omega)>0$ such that
\begin{align}\label{lem441}
\|u\|_{C^{\delta_0,\frac{\delta_0}{2}}(\bar\Omega \times[t,t+1])}+\|v\|_{C^{\delta_0,\frac{\delta_0}{2}}(\bar\Omega \times[t,t+1])}\le L_7~~for~all~ t>1.
\end{align}
\end{lemma}
{\bf Proof.} The H\"{o}lder regularity of $v$ in \eqref{lem441} can be obtain by the standard elliptic regularity
theory along with boundedness of $u$ in \eqref{lem111} and the Imbedding Theorem. Rewrite the first equation of \eqref{pp} as
\begin{align*}
u_t&=\nabla\cdot(a(x,t,u,\nabla u))+b(x,t,u),~~x\in\Omega,~t>0,
\end{align*}
where $a(x,t,u,\nabla u)=\nabla u-\frac{\chi}{1-\alpha}u\nabla z$ and $b(x,t,u)=ru-\mu u^2$.
Based on the boundedness of $u$ in \eqref{lem111} and the Young inequality, we know
\begin{align}\label{441}
a(x,t,u,\nabla u)\cdot\nabla u=|\nabla u|^2-\frac{\chi}{1-\alpha}u\nabla u\cdot\nabla z\ge \frac{1}{2}|\nabla u|^2-C_{33}|\nabla z|^2
\end{align}
with $C_{33}=\frac{\chi^2}{2(1-\alpha)^2} L_1^2$, and
\begin{align}\label{442}
|a(x,t,u,\nabla u)|\le |\nabla u|+\frac{\chi}{1-\alpha}|u\nabla z|\le |\nabla u|+C_{34}|\nabla z|
\end{align}
with $C_{34}=\frac{\chi }{1-\alpha}L_1$, and
\begin{align}\label{443}
|b(x,t,u)|\le ru+\mu u^2\le C_{35}
\end{align}
with $C_{35}=r L_1+\mu L_1^2$. Hence, a combination of \eqref{441}--\eqref{443} entails the H\"{o}lder estimate of $u$ in \eqref{lem441} based on the classical regularity estimate established in \cite[Theorem 1.3]{PV1993} due to $\nabla z\in L^\infty((0,\infty),L^\infty(\Omega))$ by \eqref{lem431}. Hence, the proof is complete. \qquad$\Box$\medskip

In order to obtain the exponential convergence of $(u,v)$, it is necessary to establish the estimate of  $\nabla u$ in $L^{q_0}$-norm for some $q_0>n$ which can be arrived by dealing with $\int_\Omega |\Delta z|^{2p}dx$.

\begin{lemma}\label{lem45}
If $\alpha\in(0,\frac{1}{2})$ with $\mu>\tilde\mu_0$, then for $p\in(1,\min\{\frac{n+2}{2},\frac{1}{2\alpha}\})$ it admits that
\begin{align}\label{lem451}
\int_\Omega |\Delta z|^{2p}dx\le L_8,~~
\end{align}
with some $L_8=L_8(\chi,\alpha,r,\mu,\Omega)>0$.
\end{lemma}
{\bf Proof.}\
Multiply \eqref{pp}$_2$ by $|\Delta z|^{2(p-1)}\Delta z$ with $p>1$ to have
\begin{align*}
0&=\int_\Omega |\Delta z|^{2(p-1)}\Delta z\big(\Delta z+\frac{\alpha}{1-\alpha}\frac{|\nabla z|^2}{z}-(1-\alpha) z+(1-\alpha) u{z^{-\frac{\alpha}{1-\alpha}}}\big)dx\nonumber\\
&=\int_\Omega |\Delta z|^{2p}dx+\frac{\alpha}{1-\alpha}\int_\Omega |\Delta z|^{2(p-1)}\Delta z\frac{|\nabla z|^2}{z}dx-(1-\alpha)\int_\Omega z|\Delta z|^{2(p-1)}\Delta zdx\nonumber\\
&~~+(1-\alpha)\int_\Omega |\Delta z|^{2(p-1)}\Delta z uz^{-\frac{\alpha}{1-\alpha}}dx.
\end{align*}
By the Young inequality, we get
\begin{align*}
\int_\Omega |\Delta z|^{2p}dx\le \frac{3}{4}\int_\Omega |\Delta z|^{2p}dx+4^{2p-1}\int_\Omega (\frac{|\nabla z|^2}{z})^{2p}dx+4^{2p-1}\int_\Omega z^{2p}dx+4^{2p-1}\int_\Omega (uz^{-\frac{\alpha}{1-\alpha}})^{2p}dx,
\end{align*}
which along with the definition of $z=v^{1-\alpha}$ shows that
\begin{align}\label{451}
\int_\Omega |\Delta z|^{2p}dx\le 4^{2p}\int_\Omega \frac{|\nabla v|^{4p}}{v^{2(1+\alpha)p}}dx+4^{2p}\int_\Omega v^{2(1-\alpha)p}dx+4^{2p}\int_\Omega \frac{u^{2p}}{v^{2\alpha p}}dx.
\end{align}
If $\alpha\in(0,\frac{1}{2})$ and $\mu>\tilde\mu_0$ with $p\in(1,\min\{\frac{1}{2\alpha},\frac{n+2}{2}\})$, it is known by the Young inequality with \eqref{lem11} and \eqref{lem131} that
\begin{align}\label{452}
\int_\Omega \frac{u^{2p}}{v^{2\alpha p}}dx\le \|u\|_{L^\infty(\Omega)}^{(1-2\alpha)p}\int_\Omega (\frac{u}{v})^{2\alpha p}dx\le 2|\Omega|L_1^{(1-2\alpha)p},
\end{align}
and
\begin{align}\label{453}
\int_\Omega \frac{|\nabla v|^{4p}}{v^{2(1+\alpha)p}}dx&=\int_\Omega \big(\frac{|\nabla v|^2}{v^2}\big)^\frac{n+2-2p}{n+1}\frac{|\nabla v|^{4p-\frac{2(n+2-2p)}{n+1}}}{v^{2p(1+\alpha)-\frac{2(n+2-2p)}{n+1}}}\nonumber\\
&\le \int_\Omega \frac{|\nabla v|^2}{v^2}dx+\int_\Omega \frac{|\nabla v|^{2(n+2)}}{v^\frac{(2(1+\alpha)(n+1)+4)p-2(n+2)}{2p-1}}dx\nonumber\\
&\le |\Omega|+\int_\Omega \frac{|\nabla v|^{2(n+2)}}{v^\frac{(2(1+\alpha)(n+1)+4)p-2(n+2)}{2p-1}}dx.
\end{align}
Let $\tilde p:=\frac{(2(1+\alpha)(n+1)+4)p-2(n+2)}{2p-1}$. Then for $\alpha\in(0,\frac{1}{2})$ and $\mu>\tilde\mu_0$ with $p\in(1,\min\{\frac{1}{2\alpha},\frac{n+2}{2}\})$, it is known that $2<\tilde p<n+3\le 2n+2$ and then by \eqref{131} that
\begin{align*}
\int_\Omega \frac{|\nabla v|^{2(n+2)}}{v^{\tilde p}}dx&\le C_1\int_\Omega \frac{u^{n+2}}{v^{\tilde p-n-2}}dx+C_2\int_\Omega v^{2(n+2)-\tilde p}dx
\end{align*}
which along with \eqref{lem111} yields that
\begin{align}\label{454}
\int_\Omega \frac{|\nabla v|^{2(n+2)}}{v^{\tilde p}}dx\le C_1L_1^{n+2}C_{36}^{n+2-\tilde p}|\Omega|+C_2C_{36}^{2n+2-\tilde p}|\Omega|
\end{align}
for the case of $\tilde p\in(2,n+2]$ via the fact that $\|v\|_{L^\infty(\Omega)}\le C_{36} $ due to the elliptic regularity theory with \eqref{lem111} with some $C_{36}=C_{36}(\chi,\alpha,r,\mu,\Omega )>0$, and by \eqref{lem131} that
\begin{align}\label{455}
\int_\Omega \frac{|\nabla v|^{2(n+2)}}{v^{\tilde p}}dx&\le C_1L_1^{2(n+2)-\tilde p}\int_\Omega (\frac{u}{v})^{\tilde p-n-2}dx+C_2C_{36}^{2(n+2)-\tilde p}|\Omega|\nonumber\\
&\le 2C_1L_1^{2(n+2)-\tilde p}|\Omega|+C_2C_{36}^{2(n+2)-\tilde p}|\Omega|
\end{align}
for the case of $\tilde p\in(n+2,n+3)$. Hence, for $\alpha\in(0,\frac{1}{2})$ and $\mu>\tilde\mu_0$ with $p\in(1,\min\{\frac{1}{2\alpha},\frac{n+2}{2}\})$, a combination of \eqref{451}--\eqref{455} shows that
\begin{align*}
\int_\Omega |\Delta z|^{2p}dx\le L_8
\end{align*}
with some $L_8=L_8(\chi,\alpha,r,\mu,\Omega)>0$, and then completes the proof of \eqref{lem451}.
\qquad$\Box$\medskip

Now, we establish the desired estimate of  $\nabla u$ in $L^{q_0}$-norm for some $q_0>n$.
\begin{lemma}\label{lem46}
If $\alpha\in(0,\frac{n+2}{n^2+4})$ with $\mu>\tilde\mu_0$, then there exists some $q_0>n$ and $L_9=L_9(\chi,\alpha,r,\mu,\Omega)>0$ such that
\begin{align}\label{lem461}
\int_\Omega |\nabla u|^{q_0}dx\le L_9,~~t>0.
\end{align}
\end{lemma}
{\bf Proof.}\
By the Young inequality with \eqref{lem431}, \eqref{lem111} and \eqref{lem451}, we know that
\begin{align}\label{461}
\frac{1}{2}\frac{d}{dt}\int_\Omega |\nabla u|^2dx&=\int_\Omega \nabla u\cdot\nabla (\Delta u-\frac{\chi}{1-\alpha}\nabla \cdot(u\nabla z)+ru-\mu u^2)dx\nonumber\\
&=-\int_\Omega |\Delta u|^2dx+\frac{\chi}{1-\alpha}\int_\Omega \Delta u(\nabla u\cdot\nabla z+u\Delta z)dx+r\int_\Omega |\nabla u|^2dx-2\mu\int_\Omega u|\nabla u|^2dx\nonumber\\
&\le \frac{\chi^2}{2(1-\alpha)^2}\int_\Omega |\nabla z|^2|\nabla u|^2dx+\frac{\chi^2}{2(1-\alpha)^2}\int_\Omega u^2|\Delta z|^2dx+r\int_\Omega |\nabla u|^2dx-2\mu\int_\Omega u|\nabla u|^2dx\nonumber\\
&\le\big(\frac{\chi^2}{2(1-\alpha)^2}L_6^2+r\big)\int_\Omega |\nabla u|^2dx-2\mu\int_\Omega u|\nabla u|^2dx+\frac{\chi^2L_1^2}{2(1-\alpha)^2}\int_\Omega |\Delta z|^2dx\nonumber\\
&\le -\frac{1}{2}\int_\Omega |\nabla u|^2dx+C_{37}\int_\Omega \frac{|\nabla u|^2}{u}dx+C_{38},~~t>0
\end{align}
with $C_{37}=\frac{1}{8\mu}(\frac{\chi^2}{2(1-\alpha)^2}L_6^2+r+\frac{1}{2})^2$ and $C_{38}=\frac{\chi^2L_1^2}{2(1-\alpha)^2}L_8$. In addition, a direct calculation with \eqref{lem111} and \eqref{lem431} shows that
\begin{align}\label{462}
\frac{d}{dt}\int_\Omega u\ln udx&=\int_\Omega \ln u(\Delta u-\frac{\chi}{1-\alpha}\nabla\cdot(u\nabla z)+ru-\mu u^2)dx+r\int_\Omega udx-\mu\int_\Omega u^2dx\nonumber\\
&\le -\int_\Omega \frac{|\nabla u|^2}{u}dx+\frac{\chi}{1-\alpha}\int_\Omega \nabla u\cdot\nabla zdx+\int_\Omega (ru\ln u+ru-\mu u^2-\mu u^2\ln u)dx\nonumber\\
&\le -\int_\Omega u\ln u dx-\frac{1}{2}\int_\Omega \frac{|\nabla u|^2}{u}dx+\frac{\chi^2}{2(1-\alpha)^2}\int_\Omega u|\nabla z|^2dx+C_{39}\nonumber\\
&\le -\int_\Omega u\ln udx-\frac{1}{2}\int_\Omega \frac{|\nabla u|^2}{u}dx+C_{40},~~t>0
\end{align}
with $C_{39}=|\Omega|\max_{s>0}((r+1)s\ln s+rs-\mu s^2-\mu s^2\ln s)$ and $C_{40}=C_{39}+L_1L_6^2|\Omega|$. Hence, we get from \eqref{461} and \eqref{462} that
\begin{align*}
\frac{d}{dt}\Big\{\frac{1}{2}\int_\Omega |\nabla u|^2dx+2C_{37}\int_\Omega u\ln udx\Big\}\le -\Big\{\frac{1}{2}\int_\Omega |\nabla u|^2dx+2C_{37}\int_\Omega u\ln udx\Big\}+C_{38}+2C_{37}C_{40}
\end{align*}
for $t>0$, which yields that
\begin{align}\label{463}
\int_\Omega |\nabla u|^2dx&\le \max\big\{2C_{38}+4C_{37}C_{40},\int_\Omega |\nabla u_0|^2dx+4C_{37}\int_\Omega u_0\ln u_0dx \big\}-4C_{37}\int_\Omega u\ln udx\nonumber\\
&\le C_{41},~~t>0
\end{align}
with some $C_{41}>0$ due to $-s\ln s\le \frac{1}{e}$ for $s>0$.

Furthermore, based on the identity $2\nabla u\cdot\nabla \Delta u=\Delta |\nabla u|^2-2|D^2u|^2$, we know for $q>1$ that
\begin{align}\label{464}
\frac{d}{dt}\int_\Omega |\nabla u|^{2q}dx&=2q\int_\Omega |\nabla u|^{2(q-1)}\nabla u\cdot\nabla (\Delta u-\frac{\chi}{1-\alpha}\nabla \cdot(u\nabla z)+ru-\mu u^2)dx\nonumber\\
&=q\int_\Omega |\nabla u|^{2(q-1)}\Delta |\nabla u|^2dx-2q\int_\Omega |\nabla u|^{2(q-1)}|D^2u|^2dx\nonumber\\
&~~-\frac{2q\chi}{1-\alpha}\int_\Omega |\nabla z|^{2(q-1)}\nabla u\cdot\nabla [\nabla\cdot(u\nabla z)]dx+2qr\int_\Omega |\nabla z|^{2q}dx-2\mu q\int_\Omega u|\nabla u|^{2q}dx\nonumber\\
&\le -\frac{4(q-1)}{q}\int_\Omega |\nabla|\nabla u|^q|^2dx-2q\int_\Omega |\nabla u|^{2(q-1)}|D^2u|^2dx+2qr\int_\Omega |\nabla z|^{2q}dx\nonumber\\
&~~+\frac{2q\chi}{1-\alpha}\int_\Omega |\nabla z|^{2(q-1)}\Delta u \nabla\cdot(u\nabla z)dx\nonumber\\
&~~+\frac{2q(q-1)\chi}{1-\alpha}\int_\Omega |\nabla z|^{2(q-2)}\nabla|\nabla u|^2\cdot\nabla u\nabla\cdot(u\nabla z)dx,~~t>0.
\end{align}
By the Young inequality with \eqref{lem111} and \eqref{lem431}, we know that
\begin{align}\label{465}
\int_\Omega |\nabla z|^{2(q-1)}\Delta u \nabla\cdot(u\nabla z)dx&\le \epsilon_1\int_\Omega |\nabla u|^{2(q-1)}|\Delta u|^2dx+\frac{1}{4\epsilon_1}\int_\Omega |\nabla u|^{2(q-1)}(\nabla u\cdot\nabla z+u\Delta z)^2dx\nonumber\\
&\le n\epsilon_1\int_\Omega |\nabla u|^{2(q-1)}|D^2 u|^2dx+\frac{1}{2\epsilon_1}\int_\Omega |\nabla u|^{2q}|\nabla z|^2dx\nonumber\\
&~~+\frac{1}{2\epsilon_1}\int_\Omega u^2|\nabla u|^{2(q-1)}|\Delta z|^2dx\nonumber\\
&\le n\epsilon_1\int_\Omega |\nabla u|^{2(q-1)}|D^2 u|^2dx+\frac{L_6^2}{2\epsilon_1}\int_\Omega |\nabla u|^{2q}dx\nonumber\\
&~~+\frac{L_1^2}{2\epsilon_1}\int_\Omega |\nabla u|^{2(q-1)}|\Delta z|^2dx,
\end{align}
due to $|\Delta u|^2\le n|D^2 u|^2$ with $\epsilon_1>0$, and
\begin{align}\label{466}
\int_\Omega |\nabla z|^{2(q-2)}\nabla|\nabla u|^2\cdot\nabla u\nabla\cdot(u\nabla z)dx&\le \epsilon_2\int_\Omega |\nabla u|^{2(q-2)}|\nabla  |\nabla u|^2|^2dx\nonumber\\
&~~+\frac{1}{4\epsilon_2}\int_\Omega |\nabla u|^{2(q-1)}(\nabla u\cdot\nabla z+u\Delta z)^2dx\nonumber\\
&\le \frac{4\epsilon_2}{q^2}\int_\Omega |\nabla|\nabla u|^q|^2dx+\frac{L_6^2}{2\epsilon_2}\int_\Omega |\nabla u|^{2q}dx\nonumber\\
&~~+\frac{L_1^2}{2\epsilon_2}\int_\Omega |\nabla u|^{2(q-1)}|\Delta z|^2dx.
\end{align}
Taking $\epsilon_1=\frac{1-\alpha}{n\chi}$ and $\epsilon_2=\frac{1-\alpha}{4\chi}$, we know from \eqref{464}--\eqref{466} and \eqref{lem451} for $p_0\in(1,\min\{\frac{n+2}{2},\frac{1}{2\alpha}\})$ that
\begin{align}\label{467}
\frac{d}{dt}\int_\Omega |\nabla u|^{2q}dx&\le -\frac{2(q-1)}{q}\int_\Omega |\nabla|\nabla u|^q|^2dx+C_{42}\int_\Omega |\nabla u|^{2q}dx+C_{43}\int_\Omega |\nabla u|^{2(q-1)}|\Delta z|^2dx\nonumber\\
&\le -\frac{2(q-1)}{q}\int_\Omega |\nabla|\nabla u|^q|^2dx+C_{42}\int_\Omega |\nabla u|^{2q}dx+C_{43}\int_\Omega |\nabla u|^\frac{2(q-1)p_0}{p_0-1}dx+C_{43}\int_\Omega |\Delta z|^{2p_0}dx\nonumber\\
&\le -\frac{2(q-1)}{q}\int_\Omega |\nabla|\nabla u|^q|^2dx+C_{42}\int_\Omega |\nabla u|^{2q}dx+C_{43}\int_\Omega |\nabla u|^\frac{2(q-1)p_0}{p_0-1}dx+C_{44}
\end{align}
with $C_{42}=2qr+\frac{q\chi^2L_6^2}{(1-\alpha)^2}(n+4q-4)$, $C_{43}=\frac{q\chi^2L_1^2}{(1-\alpha)^2}(n+4q-4)$ and $C_{44}=C_{43}L_8$ for $t>0$. Applying the Gaglirado-Nirenberg inequality and the Poincar\'{e} inequality with \eqref{463}, we get for $q>1$ that
\begin{align}\label{468}
\int_\Omega |\nabla u|^{2q}dx&=\||\nabla u|^q\|_{L^2(\Omega)}^2\nonumber\\
&\le C_{GN}\|\nabla |\nabla u|^q\|_{L^2(\Omega)}^{2a_1}\||\nabla u|^q\|_{L^\frac{2}{q}(\Omega)}^{2(1-a_1)}+C_{GN}\||\nabla u|^q\|_{L^\frac{2}{q}(\Omega)}^2\nonumber\\
&\le \epsilon_3\int_\Omega |\nabla|\nabla u|^q|^2dx+C(\epsilon_3),
\end{align}
due to $a_1=\frac{\frac{qn}{2}-\frac{n}{2}}{1-\frac{n}{2}+\frac{qn}{2}}\in(0,1)$, with $\epsilon_3>0$ and $C(\epsilon_3)=\epsilon_3^{-\frac{a_1}{1-a_1}}C_{GN}^\frac{1}{1-a_1}C_{41}^q+C_{GN}C_{41}^q$. If $q\in(1,\min\{\frac{n+2-4\alpha}{2\alpha n},\frac{n+4}{2}\})$ with  $p_0\in\big(\frac{2+qn}{2+n},\min\{\frac{n+2}{2},\frac{1}{2\alpha},\frac{1}{(2-q)_+}\}\big)$, then we know
$$a_2=\frac{\frac{qn}{2}-\frac{q(p_0-1)n}{2(q-1)p_0}}{1-\frac{n}{2}+\frac{qn}{2}}\in(0,1)~~{\rm and}~~\frac{(q-1)p_0}{q(p_0-1)}a_2<1,$$
and then by the Gaglirado-Nirenberg inequality and the Poincar\'{e} inequality with \eqref{463} that
\begin{align}\label{469}
\int_\Omega |\nabla u|^\frac{2(q-1)p_0}{p_0-1}dx&=\||\nabla u|^q\|_{L^\frac{2(q-1)p_0}{q(p_0-1)}(\Omega)}^\frac{2(q-1)p_0}{q(p_0-1)}\nonumber\\
&\le C_{GN}\|\nabla |\nabla u|^q\|_{L^2(\Omega)}^{\frac{2(q-1)p_0}{q(p_0-1)}a_2}\||\nabla u|^q\|_{L^\frac{2}{q}(\Omega)}^{\frac{2(q-1)p_0}{q(p_0-1)}(1-a_2)}+C_{GN}\||\nabla u|^q\|_{L^\frac{2}{q}(\Omega)}^\frac{2(q-1)p_0}{q(p_0-1)}\nonumber\\
&\le \epsilon_4\int_\Omega |\nabla|\nabla u|^q|^2dx+C(\epsilon_4)
\end{align}
with $\epsilon_4>0$ and $C(\epsilon_4)=\epsilon_4^{-\frac{{(q-1)p_0}a_2}{q(p_0-1)-{(q-1)p_0}a_2}}
C_{GN}^\frac{q(p_0-1)}{q(p_0-1)-{(q-1)p_0}a_2}C_{41}^\frac{(q-1)p_0(1-a_2)}{q(p_0-1)-{(q-1)p_0}a_2}
+C_{GN}C_{41}^\frac{(q-1)p_0}{p_0-1}$. Selecting $\epsilon_3=\frac{q-1}{(C_{42}+1)q}$ and $\epsilon_4=\frac{q-1}{C_{43}q}$, then for $q\in(1,\min\{\frac{n+2-4\alpha}{2\alpha n},\frac{n+4}{2}\})$ we obtain from \eqref{467}, \eqref{468} and \eqref{469} that
\begin{align}\label{4610}
\frac{d}{dt}\int_\Omega |\nabla u|^{2q}dx&\le-\int_\Omega |\nabla u|^{2q}dx+C_{45},~~t>0,
\end{align}
with $C_{45}=(C_{42}+1)C(\frac{q-1}{(C_{42}+1)q})+C_{43}C(\frac{q-1}{C_{43}q})+C_{44}$, which yields
\begin{align*}
\int_\Omega |\nabla u|^{2q}dx\le C_{46},~~t>0
\end{align*}
with $C_{46}=\max\{C_{45},\int_\Omega |\nabla u_0|^{2q}dx\}$. Consequently, if $\alpha\in(0,\frac{n+2}{n^2+4})$ with $\mu>\tilde\mu_0$, then for $q_0\in(n, \min\{\frac{n+2-4\alpha}{\alpha n},{n+4}\})$ the estimate \eqref{lem461} is valid with some $L_9=L_9(\chi,\alpha,r,\mu,\Omega)>0$. \qquad$\Box$\medskip

\section{Long-time behavior for $\alpha\in(0,\frac{1}{2})$}

Let $a:=\frac{r}{\mu}$. We at first estimate $\frac{d}{dt}\int_\Omega (u-a-a\ln \frac{u}{a})dx$.
\begin{lemma}\label{lem51}
For $\alpha\in(0,1)$ and $\mu>\mu_0$, it holds that
\begin{align}\label{lem511}
\frac{d}{dt}\int_\Omega (u-a-a\ln \frac{u}{a})dx\le-\frac{a}{2}\int_\Omega \frac{|\nabla u|^2}{u^2}dx+\frac{\chi^2a}{2(1-\alpha)^2}\int_\Omega{|\nabla z|^2}dx-\mu\int_\Omega (u-a)^2dx
\end{align}
for $t>0$.
\end{lemma}
{\bf Proof.}\
A direct calculation with the Young inequality shows that
\begin{align*}
\frac{d}{dt}\int_\Omega (u-a-a\ln \frac{u}{a})dx&=\int_\Omega \frac{u-a}{u}(\Delta u-\frac{\chi}{1-\alpha}\nabla\cdot(u\nabla z)+ru-\mu u^2)dx\nonumber\\
&=-a\int_\Omega \frac{|\nabla u|^2}{u^2}dx+\frac{\chi a}{1-\alpha}\int_\Omega \frac{\nabla u}{u}\cdot \nabla zdx-\mu\int_\Omega(u-a)^2dx\nonumber\\
&\le -\frac{a}{2}\int_\Omega \frac{|\nabla u|^2}{u^2}dx+\frac{\chi^2a}{2(1-\alpha)^2}\int_\Omega |\nabla z|^2dx-\mu\int_\Omega(u-a)^2dx,~~t>0.
\end{align*}
This completes the proof. \qquad$\Box$\medskip

When $\alpha\in(0,\frac{1}{2})$, we can control $\int_\Omega |\nabla z|^2dx$ by $\int_\Omega \frac{|\nabla u|^2}{u^2}dx$ after some time based on the qualitative estimate of $\|u\|_{L^\infty(\Omega)}$ with respect to $\mu$.

\begin{lemma}\label{lem53}
If $\alpha\in(0,\frac{1}{2})$, there exists some $\mu_\star\ge \mu_0$ such that
\begin{align}\label{lem531}
\int_\Omega |\nabla z(x,t)|^2dx\le \frac{(1-\alpha)^2}{\chi^2}\int_\Omega \frac{|\nabla u(x,t)|^2}{u(x,t)^2}dx,
\end{align}
provided $\mu>\mu_\star$ for $t>t_0$ with $t_0$ determined in Theorem \ref{th1}.
\end{lemma}
{\bf Proof.}\
It is known by \eqref{4110} and the convexity of $\Omega$ that
\begin{align}\label{531}
2(1-\alpha)\int_\Omega |\nabla z|^2dx&\le -2\int_\Omega |D^2z|^2dx-\frac{2\alpha}{1-\alpha}\int_\Omega\frac{|\nabla z|^4}{z^2}dx\nonumber\\
&~~+ \frac{2\alpha}{1-\alpha}\int_\Omega\nabla z\cdot\frac{\nabla|\nabla z|^2}{z}dx+2(1-\alpha)\int_\Omega \nabla z\cdot\nabla (uz^{-\frac{\alpha}{1-\alpha}})dx
\end{align}
Since $\nabla|\nabla z|^2=2 D^2z\cdot\nabla z$, we get by the Young inequality that
\begin{align}\label{532}
\frac{2\alpha}{1-\alpha}\int_\Omega\nabla z\cdot\frac{\nabla|\nabla z|^2}{z}dx\le 2\int_\Omega |D^2 z|^2dx+\frac{2\alpha^2}{(1-\alpha)^2}\int_\Omega \frac{|\nabla z|^4}{z^2}dx.
\end{align}
In addition, for $\alpha\in(0,\frac{1}{2})$ and $\mu>\mu_0$, we get by the Young inequality with \eqref{244} that
\begin{align}\label{533}
2(1-\alpha)\int_\Omega \nabla z\cdot\nabla (uz^{-\frac{\alpha}{1-\alpha}})dx&=2(1-\alpha)\int_\Omega \frac{\nabla u\cdot\nabla z}{z^\frac{\alpha}{1-\alpha}}dx-2\alpha\int_\Omega u\frac{|\nabla z|^2}{z^\frac{1}{1-\alpha}}dx\nonumber\\
&\le \frac{(1-\alpha)^2}{2\chi^2}\int_\Omega \frac{|\nabla u|^2}{u^2}dx+{2\chi^2}\int_\Omega u^2\frac{|\nabla z|^2}{z^\frac{2\alpha}{1-\alpha}}dx-2\alpha \int_\Omega u\frac{|\nabla z|^2}{z^\frac{1}{1-\alpha}}dx\nonumber\\
&\le \frac{(1-\alpha)^2}{2\chi^2}\int_\Omega \frac{|\nabla u|^2}{u^2}dx+2\chi^2\|u\|_{L^\infty(\Omega)}^{2(1-\alpha)}\int_\Omega u^{2\alpha}\frac{|\nabla z|^{2}}{z^\frac{2\alpha}{1-\alpha}}dx-2\alpha \int_\Omega u\frac{|\nabla z|^2}{z^\frac{1}{1-\alpha}}dx\nonumber\\
&\le\frac{(1-\alpha)^2}{2\chi^2}\int_\Omega \frac{|\nabla u|^2}{u^2}dx+\frac{2\chi^2 L_0^{2(1-\alpha)}}{\mu^{2(1-\alpha)}}\int_\Omega u^{2\alpha}\frac{|\nabla z|^{2}}{z^\frac{2\alpha}{1-\alpha}}dx-2\alpha \int_\Omega u\frac{|\nabla z|^2}{z^\frac{1}{1-\alpha}}dx\nonumber\\
&\le \frac{(1-\alpha)^2}{2\chi^2}\int_\Omega \frac{|\nabla u|^2}{u^2}dx+\frac{C_{47}}{\mu^\frac{2(1-\alpha)}{1-2\alpha}}\int_\Omega |\nabla z|^2dx
\end{align}
with $C_{47}=(\frac{1}{2\alpha})^\frac{2\alpha}{1-2\alpha}
(2\chi^2)^\frac{1}{1-2\alpha}L_0^\frac{2(1-\alpha)}{1-2\alpha}$ for $t>t_0$.
Let $\mu_\star=\max\{\mu_0, (2C_{47})^\frac{1-2\alpha}{2(1-\alpha)}\}$. Then for $\mu>\mu_\star$, it is known that
\begin{align*}
2(1-\alpha)-\frac{C_{47}}{\mu^\frac{2(1-\alpha)}{1-2\alpha}}\ge \frac{1}{2},
\end{align*}
and then by the combination of \eqref{531}--\eqref{533} that
\begin{align*}
\int_\Omega |\nabla z|^2dx\le \frac{(1-\alpha)^2}{\chi^2}\int_\Omega \frac{|\nabla u|^2}{u^2}dx
\end{align*}
for $t>t_0$. This completes the proof of \eqref{lem531}. \qquad$\Box$\medskip

Now, we establish that $(u,v)$ converges to $(\frac{r}{\mu},\frac{r}{\mu})$ as $t\rightarrow\infty$ if $\alpha\in(0,\frac{1}{2})$ and $\mu>0$ sufficiently large.
\begin{lemma}\label{lem54}
For $\alpha\in(0,\frac{1}{2})$ and $\mu>\mu_\star$, then
\begin{align}\label{lem541}
\|u-\frac{r}{\mu}\|_{L^\infty(\Omega)}+\|v-\frac{r}{\mu}\|_{L^\infty(\Omega)}\rightarrow 0,~~as~t\rightarrow\infty.
\end{align}
\end{lemma}
{\bf Proof.}\
For $\alpha\in(0,\frac{1}{2})$ and $\mu>\mu_\star$, we obtain from \eqref{lem511} and \eqref{lem531} that
\begin{align}\label{541}
\frac{d}{dt}\int_\Omega (u-a-a\ln \frac{u}{a})dx\le-\mu\int_\Omega (u-a)^2dx,~~t>t_0,
\end{align}
which integrating from $t_0$ to $t$ yields
\begin{align*}
\int_{t_0}^t\int_\Omega(u-a)^2dxds\le \int_\Omega \big(u(\cdot,t_0)-a-a\ln\frac{u(\cdot,t_0)}{a}\big)dx\le C_{48},~~t>t_0
\end{align*}
with some $C_{48}=C_{48}(r,\mu)>0$, and then by the Beppo Levi theorem that
\begin{align*}
\int_{t_0}^\infty\int_\Omega (u-a)^2dxds\le C_{48}
\end{align*}
This together with the Schauder estimate of $u$ in \eqref{lem441} entails that
\begin{align}\label{542}
\|u-a\|_{L^\infty(\Omega)}\rightarrow 0~~as~t\rightarrow\infty.
\end{align}
and moveover that $v\rightarrow \frac{r}{\mu}$ in $L^\infty(\Omega)$ as $t\rightarrow\infty$ by the Maximum principle. The proof of \eqref{lem541} is complete. \qquad$\Box$\medskip

If $\alpha\in(0,\frac{n+2}{n^2+4})$ and $\mu>0$ sufficiently large, it can be proved that the solution $(u,v)$ furthermore enjoys the exponential convergence.
\begin{lemma}\label{lem55}
If $\alpha\in(0,\frac{n+2}{n^2+4})$ and $\mu>\mu_\star$, there exists some $\eta_\star>0$ and $L_\star=L_\star(\chi,\alpha,r,\mu,\Omega)>0$ such that
\begin{align}\label{lem551}
\|u-a\|_{L^\infty(\Omega)}\le L_\star e^{-\eta_\star t},~~t>0.
\end{align}
\end{lemma}
{\bf Proof.}\ A direct calculation tells that
\begin{align*}
\lim_{u\rightarrow a}\frac{u-a-a\ln\frac{u}{a}}{(u-a)^2}=\frac{1}{2a}.
\end{align*}
This together with \eqref{lem541} entails that there exists $t_5\ge t_0$ such that
\begin{align}\label{551}
\frac{1}{4a}(u-a)^2\le u-a-a\ln\frac{u}{a}\le \frac{1}{a}(u-a)^2,~~t>t_5
\end{align}
and then by \eqref{541} that
\begin{align*}
\frac{d}{dt}\int_\Omega (u-a-a\ln \frac{u}{a})dx\le-r\int_\Omega (u-a-a\ln \frac{u}{a})dx,~~t>t_5.
\end{align*}
Hence, we get
\begin{align*}
\int_\Omega (u-a-a\ln\frac{u}{a})dx\le C_{49}e^{-r(t-t_5)},~~t>t_5,
\end{align*}
with $C_{49}=\int_\Omega (u(\cdot,t_5)-a-a\ln\frac{u(\cdot,t_5)}{a})dx$, which along with \eqref{551} yields that
\begin{align}\label{552}
\|u-a\|_{L^2(\Omega)}\le C_{50}e^{-r(t-t_5)},~~t>t_5
\end{align}
with  $C_{50}=4aC_{44}$. Applying the Gaglirado-Nirenberg inequality and the Poincar\'{e} inequality, we get by \eqref{552} and \eqref{lem461} that
\begin{align*}
\|u-a\|_{L^\infty(\Omega)}&\le C_{GN}\|\nabla u\|_{L^{q_0}(\Omega)}^\frac{nq_0}{nq_0+2(q_0-n)}\|u-a\|_{L^2(\Omega)}^\frac{2(q_0-n)}{nq_0+2(q_0-n)}
+C_{GN}\|u-a\|_{L^2(\Omega)}\nonumber\\
&\le  C_{GN}C_{50}L_9^\frac{nq_0}{nq_0+2(q_0-n)}e^{-\frac{2(q_0-n)}{nq_0+2(q_0-n)}r(t-t_5)}+C_{GN}C_{50}e^{-r(t-t_5)}  \nonumber\\
&\le C_{51}e^{{-\eta_\star}(t-t_5)},~~t>t_5.
\end{align*}
with $\eta_\star=\frac{2(q_0-n)r}{nq_0+2(q_0-n)}$ and $C_{51}=C_{GN}C_{50}(1+L_9^\frac{nq_0}{nq_0+2(q_0-n)})$, and then
\begin{align}\label{553}
\|u-a\|_{L^\infty(\Omega)}\le C_{52}e^{{-\eta_\star}t},~~t>0.
\end{align}
with $C_{52}=e^{\eta_\star t_5}\max\{L_0+a,C_{52}\}$.
In addition,  this together with the Maximum principle entails that there exists some $C_{53}>0$ such that
 \begin{align}\label{553}
\|v-a\|_{L^\infty(\Omega)}\le C_{53}e^{{-\eta_\star}t},~~t>0.
\end{align}
 Consequently, the proof of \eqref{lem551} is complete with $L_\star=\max\{C_{52},C_{53}\}$.
\qquad$\Box$\medskip

{\bf Proof of Theorem \ref{th2}}\
The proof is contained in Lemmas \ref{lem54} and \ref{lem55}. \qquad$\Box$\medskip

{\noindent\bf Conflict of interest}

There is no conflict of interest.

{\noindent\bf Data availability}

Data will be made available by reasonable request.

{\noindent \bf Acknowledgements}

This work is supposed by the National Natural Science Foundation of China (No. 12201276).



{\small \begin{thebibliography}{}

\bibitem{KS1970} E.F. Keller, L.A. Segel, Initiation of slime mold aggregation viewed as an instability, J. Theoret. Biol. 26 (1970) 399--415. \\[-17pt]
\bibitem{KS1971} E.F. Keller, L.A. Segel, Traveling bands of chemotactic bacteria: a theoretical analysis, J. Theoret. Biol. 30 (1971) 235--248. \\[-17pt]
\bibitem{NSY1997} T. Nagai, T. Senba, K. Yoshida, Application of the Trudinger-Moser inequality to a parabolic system of chemotaxis,
Funkcial. Ekvac. 40 (1997) 411--433.\\[-17pt]
\bibitem{C2015} X.R. Cao, Global bounded solutions of the higher-dimensional Keller-Segel system under smallness conditions in optimal spaces, Discrete Contin. Dynam. Syst. Ser. A 35 (2015) 1891--1904.\\[-17pt]
\bibitem{TW2007} J.I. Tello, M. Winkler, A chemotaxis system with logistic
source,  Commun. Partial Differ. Equ., 32 (2007) 849--877.\\[-17pt]
\bibitem{MW2010} M. Winkler,  Boundedness in the higher-dimensional parabolic-parabolic chemotaxis system with logistic source, Commun. Partial Differ. Equ. 35 (2010) 1516--1537. \\[-17pt]

\bibitem{C2017} X.R. Cao, Large time behavior in the logistic Keller-Segel model via maximal Sobolev regularity, Discrete Contin.
Dyn. Syst., Ser. B 22 (9) (2017) 3369--3378.\\[-17pt]
\bibitem{N1995} T. Nagai, Blow-up of radially symmetric solutions to a chemotaxis system, Adv. Math. Sci. Appl. 5 (1995) 581--601.\\[-17pt]
\bibitem{N2001} T. Nagai, Blow-up of nonradial solutions to parabolic-elliptic systems modeling chemotaxis in two dimensional domains, J. Inequal. Appl. 6, 37--51 (2001).\\[-17pt]
\bibitem{W2013} M. Winkler, Finite-time blow-up in the higher-dimensional parabolic-parabolic Keller-Segel system, J. Math. Pures Appl. 100 (2013) 748--767.\\[-17pt]
\bibitem{W2018} M. Winkler, Finite-time blow-up in low-dimensional Keller-Segel systems with
logistic-type superlinear degradationm, Z. Angew. Math. Phys. (2018) 69:40. \\[-17pt]
\bibitem{FWY2015} K. Fujie, M. Winkler, T. Yokota, Boundedness of solutions to parabolic-elliptic Keller-Segel systems with signal-dependent sensitivity, Math. Methods. Appl. Sci. 38 (2015) 1212--1224.\\[-17pt]
\bibitem{FS2016} K. Fujie and T. Senba, Global existence and boundedness in a parabolic-elliptic Keller-Segel
system with general sensitivity, Discrete Contin. Dyn. Syst. Ser. B 21 (2016) 81--102. \\[-17pt]
\bibitem{HK2025} H.I. Kurt, Improvement of criteria for global boundedness in a minimal
parabolic-elliptic chemotaxis system with singular sensitivity, Applied Mathematics Letters 167 (2025) 109570. \\[-17pt]
\bibitem{A2019} J. Ahn, Global well-posedness and asymptotic stabilization for
chemotaxis system with signal-dependent sensitivity, J. Differ. Equ. 266 (2019) 6866--6904. \\[-17pt]
\bibitem{L2026} M. Le, An improvement toward global boundedness in a fully parabolic chemotaxis with singular sensitivity in any dimension, Nonlinear Anal. 268 (2026) 114082.  \\[-17pt]
\bibitem{AKL2019} J. Ahna, K. Kang, J. Lee, Eventual smoothness and stabilization of global weak solutions in
parabolic-elliptic chemotaxis systems with logarithmic sensitivity, Nonlinear Anal. Real World Appl. 49 (2019) 312--330. \\[-17pt]
\bibitem{WY2018} M. Winkler, T. Yokota, Stabilization in the logarithmic Keller-Segel system, Nonlinear Analysis 170 (2018) 123--141. \\[-17pt]
\bibitem{LX2026} B. Li, L. Xie, The role of the general singular sensitivity in solvability
and eventual smoothness to a 2-D chemotaxis system, J. Differ. Equ.  453 (2026) 113901. \\[-17pt]

\bibitem{NS1997} T. Nagai, T. Senba, Behavior of radially symmetric solutions of a system related to chemotaxis,
Nonlinear Anal. 30 (1997) 3837--3842. \\[-17pt]
\bibitem{FS2018} K. Fujie, T. Senba, A sufficient condition of sensitivity functions for
boundedness of solutions to a parabolic-parabolic
chemotaxis system, Nonlinearity 31 (2018) 1639--1672.\\[-17pt]
\bibitem{W2022} M. Winkler, Unlimited growth in logarithmic Keller-Segel systems, J. Differ. Equ. 309 (2022) 74--97. \\[-17pt]
\bibitem{FWY2014} K. Fujie, M. Winkler, T. Yokota, Blow-up prevention by logistic sources in a parabolic-elliptic Keller-Segel system with singular sensitivity, Nonliear Anal. 109 (2014) 56--71. \\[-17pt]
\bibitem{HW2021} H.I. Kurt, W.X. Shen, Finite-time blow-up prevention by lositic source in parabolic-elliptic chemotaxis models with singular sensitivity in any dimensional setting, Siam J. Math. Anal. 53 (2021) 973--1003. \\[-17pt]
\bibitem{ZZ2017} X. D. Zhao, S. N. Zheng, Global boundedness to a chemotaxis system
with singular sensitivity and logistic source, Z. Angew. Math. Phys. 68:2 (2017) 13 pp.

\bibitem{ZZ2019} X. D. Zhao, S. N. Zheng, Global existence and boundedness of solutions to
a chemotaxis system with singular sensitivity and
logistic-type source, J. Differ. Equ. 267 (2019) 826--865.\\[-17pt]
\bibitem{CWY2016} J.H. Cao, W. Wang, H. Yu, Asymptotic behavior of solutions to two-dimensional chemotaxis
system with logistic source and singular sensitivity, J. Math. Anal. Appl. 436 (2016) 382-392.\\[-17pt]
\bibitem{LL2021} J.Q. Li, Z.P. Li, Large time behavior of solutions to a chemotaxis
system with singular sensitivity and logistic source,  Mathematische Nachrichten 294 (2021) 1374--1383. \\[-17pt]
\bibitem{Z2022} X.D. Zhao, Boundedness to a parabolic-parabolic singular chemotaxis
system with logistic source, J. Differ. Equ.  338 (2022) 388--414. \\[-17pt]
\bibitem{Z2023} X.D. Zhao, Boundedness in a logistic chemotaxis system with weakly singular sensitivity in dimension two, Nonlinearity 36 (2023) 3909--3938.\\[-17pt]
\bibitem{K2025} H.I. Kurt, Boundedness in a chemotaxis system with weak
singular sensitivity and logistic kinetics in any dimensional setting, J. Differ. Equ. 416 (2025) 1429--1461. \\[-17pt]
\bibitem{L2025} M. Le, Boundedness in a chemotaxis system with weakly singular sensitivity in dimension two with arbitrary sub-quadratic degradation sources, J. Math. Anal. Appl. 542 (2025) No. 128803. \\[-17pt]
\bibitem{LK2025} M. Le, H.I. Kurt, Global boundedness in a chemotaxis-growth system with weak
singular sensitivity in any dimensional setting, Nonlinear Anal. RWA 86 (2025) 104392. \\[-17pt]
\bibitem{K2026} H.I. Kurt, Large-time dynamics of solutions in a logistic chemotaxis system with weak singular sensitivity: uniform boundedness, pointwise persistence and stability, Discrete and Continuous Dynamical Systems 52 (2026) 255--292. \\[-17pt]
\bibitem{ZMT2025} J. Zhang, C.L. Mu, X.Y. Tu, Global existence, boundedness and large
time behavior for a chemotaxis model with
singular sensitivity and nonlinear signal
production, Nonlinear Differ. Equ. Appl. (2025) 32:54. \\[-17pt]
\bibitem{PV1993} M.M. Porzio, V. Vespri, H\"{o}lder estimates for local solutions of some doubly nonlinear degenerate parabolic equations, J. Differ. Equ. 103 (1993) 146--178. \\[-17pt]
\bibitem{W2010} M. Winkler, Aggregation vs. global diffusive behavior in the higher-dimensional Keller-Segel
model, J. Differ. Equ. 248 (2010) 2889--2905.\\[-17pt]
















\end{thebibliography}}
\end{document}